\documentclass[11pt]{amsart}
\usepackage{amsbsy,amsmath,amsfonts,relsize,amsthm,enumerate,enumitem,amssymb,color,hyperref}
\newtheorem{Theorem}{Theorem}
\newtheorem{Proposition}[Theorem]{Proposition}
\newtheorem{Corollary}[Theorem]{Corollary}
\newtheorem{Lemma}[Theorem]{Lemma}

\newtheorem{COMP}[Theorem]{Compactness Property}
\newtheorem{COMP2}[Theorem]{Compactness Property}
\newtheorem{SMOO}[Theorem]{Smooth Flow Lemma}
\newtheorem{OUT}[Theorem]{Outward Optimising Property}
\newtheorem{REG}[Theorem]{Regularity Theorem}
\newenvironment{jumpproof}{\trivlist\item[]\emph{Proof of Proposition \ref{jumpprop}}:}%
{\unskip\nobreak\hskip 1em plus 1fil\nobreak$\Box$
\parfillskip=0pt%
\endtrivlist}
\newenvironment{outwardproof}{\trivlist\item[]\emph{Proof of Outward Optimising Property \ref{outward1}}:}%
{\unskip\nobreak\hskip 1em plus 1fil\nobreak$\Box$
\parfillskip=0pt%
\endtrivlist}

{\unskip\nobreak\hskip 1em plus 1fil\nobreak$\Box$
\parfillskip=0pt%
\endtrivlist}
\newenvironment{MOTS1Proof}{\trivlist\item[]\emph{Proof of Proposition \ref{MOTS1}}:}%
{\unskip\nobreak\hskip 1em plus 1fil\nobreak$\Box$
\parfillskip=0pt%
\endtrivlist}
 \def \vs { \left ( \begin {array} {c} }
    \def \ve { \end {array} \right )} 
\newtheorem{Definition}[Theorem]{Definition}

\title{Evolving hypersurfaces by their inverse null mean curvature}
\begin{document}
\author{Kristen Moore}
\email{kristen.moore@aei.mpg.de}
\maketitle
\begin{abstract} 
We introduce a new geometric evolution equation for hypersurfaces in asymptotically flat spacetime initial data sets, that unites the theory of marginally outer trapped surfaces (MOTS) with the study of inverse mean curvature flow in asymptotically flat Riemannian manifolds. A theory of weak solutions is developed using level-set methods and an appropriate variational principle. This new flow has a natural application as a variational-type approach to constructing MOTS, and this work also gives new insights into the theory of weak solutions of inverse mean curvature flow.\\
\end{abstract}
In what follows we consider an initial data set $(M^{n+1},g,K)$ that arises as a spacelike hypersurface $M^{n+1}$ in a Lorentzian spacetime $(L^{n+2},h)$, with induced metric $g$ and second fundamental form tensor $K$. We further assume that the initial data set $(M,g,K)$ is asymptotically flat, that is, there exists a compact set $\Omega\subset M$ such that $M\backslash\Omega$ consists of a finite number of components, each diffeomorphic to $\mathbb{R}^{n+1}\backslash \bar{B}(0,1)$ and such that under these diffeomorphisms, the metric tensor $g$ and second fundamental form $K$ satisfy
\begin{align*}
|g_{ij}-\delta_{ij}|\leq \dfrac{C}{|x|^{n-1}},&\quad |g_{ij,k}|\leq\dfrac{C}{|x|^{n}},\\
 |K_{ij}|\leq\dfrac{C}{|x|^n},\quad|K_{ij,k}|\leq &\dfrac{C}{|x|^{n+1}},\quad \Bigl|\sum_iK_{ii}\Bigr|\leq\dfrac{C}{|x|^{\,\,\frac{n+3}{2}}}.
\end{align*}
as $|x|\to\infty$, where the derivatives are taken with respect to the Euclidean metric.\\

Let $\vec{n}$ denote the future directed timelike unit normal vector field of $M\subset L$, and consider a 2-sided hypersurface $\Sigma^{n}\subset M^{n+1}$ with globally defined outer unit normal vector field $\nu$ in $M$. The mean curvature vector of $\Sigma$ inside the spacetime $L$ is then given by 
\begin{equation*}\vec{H}_{\Sigma}:=H\nu-P\vec{n},
\end{equation*} 
where $H:=\text{div}_{\Sigma}(\nu)$ denotes the mean curvature of $\Sigma$ in $M$, and $P:=\text{tr}_{\Sigma}K$ is the trace of $K$ over the tangent space of $\Sigma$.

The new initial value problem is then defined as follows. Given a smooth hypersurface immersion $F_{0}:\Sigma\rightarrow M$, the evolution of $\Sigma_{0}:=F_{0}(\Sigma)$ by inverse null mean curvature is the one-parameter family of smooth immersions $F:\Sigma\times[0,T)\rightarrow M$ satisfying
\begin{equation}\tag{$*$}
\left\lbrace
\begin{aligned}
 \frac{\partial F}{\partial t}(x,t) & = \dfrac{\nu}{H+P}(x,t), \quad x\in\Sigma,t\geq0, \\
                         F(\cdot,0) & = F_{0}.
\end{aligned}
\right.
\end{equation}
The quantity $H+P$ corresponds to the null expansion or \textit{null mean curvature} $\theta_{\Sigma_t}^+$ of $\Sigma_t:=F(\Sigma,t)$ with respect to its future directed outward null vector field $l^+:=\nu+\vec{n}$,
\begin{equation*}
\theta^+_{\Sigma_t}:=\langle\vec{H}_{\Sigma_t},l^+\rangle_h=H+P,
\end{equation*}
and we assume that $(H+P)|_{_{\Sigma_0}}>0$ so that $(*)$ is parabolic and the surface $\Sigma_{t}$ expands under the flow. This flow is a generalisation of inverse mean curvature flow, which corresponds to the special time-symmetric case of $(*)$ where $K\equiv0$. Analogous to inverse mean curvature flow, in general it is expected that the null mean curvature of solutions of $(*)$ will tend to zero at some points, and that singularities will develop. Therefore, the main part of this work is devoted to developing a theory of weak solutions of the classical flow $(*)$. 

The motivation for introducing this particular generalisation of inverse mean curvature flow follows from the study of black holes and mass/energy inequalities in general relativity. In particular, it is hoped that this new flow will help to gain insight into the long standing Penrose conjecture in general relativity, which generalises the Riemannian Penrose inequality, proven by Huisken and Ilmanen \cite{HI01} using their theory of weak solutions to inverse mean curvature flow, (see \cite{B01} for an alternative proof by Bray, which applies to the case of multiple horizons).

More specifically, this flow is motivated by the theory of marginally outer trapped surfaces in general relativity. Physically, the outward null mean curvature $\theta^+_{\Sigma}$ measures the divergence of the outward directed light rays emanating from $\Sigma$. If $\theta_{\Sigma}^+$ vanishes on all of $\Sigma$, then $\Sigma$ is called a \textit{marginally outer trapped surface}, or MOTS for short. MOTS play the role of apparent horizons, or quasi-local black hole boundaries in general relativity, and are particularly useful for numerically modelling the dynamics and evolution of black holes.

From a mathematical point of view, MOTS are the Lorentzian analogue of minimal surfaces. However, since MOTS are not stationary solutions of an elliptic variational problem, the direct method of the calculus of variations is not a viable approach to the existence theory. One successful approach to proving existence of MOTS comes from studying the blow-up set of solutions of \textit{Jang's equation}
\begin{equation}\label{Jang}
\left(g^{ij}-\dfrac{\nabla^iw\nabla^jw}{|\nabla w|^2+1}\right)\left(\dfrac{\nabla_i\nabla_jw}{\sqrt{|\nabla w|^2+1}}+K_{ij}\right)=0,
\end{equation}
for the height function $w$ of a hypersurface, which was an essential ingredient in the Schoen-Yau proof of the positive mass theorem \cite{SY81}. In their analysis, Schoen and Yau showed that the boundary of the blow-up set  of Jang's equation consists of marginally trapped surfaces. Building upon this work, existence of MOTS in compact data sets with two boundary components, such that the inner boundary is (outer) trapped and the outer boundary is (outer) untrapped, has been proven by Andersson and Metzger \cite{AM09}, and independently by Eichmair \cite{E09}, using different approaches for the barrier argument at the boundary. Similarly, we see below that Jang's equation also plays a key role in the existence theory for weak solutions of $(*)$. \\

To develop the weak formulation for the classical evolution $(*$), we use the level-set method and assume the evolving surfaces are given by the level-sets,
\begin{equation}\label{levelset}\Sigma_{t}= \partial\{ x\in M \,\, \big| \,\, u(x)<t \}, \end{equation}
of a scalar function $u:M\rightarrow\mathbb{R}$. Then whenever $u$ is smooth and $\nabla u\neq0$, the surface flow equation $(*)$ is equivalent to the following degenerate elliptic scalar PDE
\begin{equation}\tag{$**$}
 \mbox{div}_M\left( \frac{\nabla u}{|\nabla u|}\right) +\left(g^{ij}-\dfrac{\nabla^iu\nabla^ju}{|\nabla u|^2}\right)K_{ij}= |\nabla u|.
\end{equation}
However, since this is an expanding flow,  the function $u$ is monotone non-decreasing. Therefore it only makes sense to study $(**)$ on initial data sets $(M,g,K)$ satisfying $\text{tr}_MK\geq0$ outside $\Sigma_0$, so that the zero function is a subsolution barrier for the Dirichlet problem for $(**)$.
\medskip

In order to solve $(**)$, we employ the method of \textit{elliptic regularisation}, and study solutions, $u_{\varepsilon}$,  of the following strictly elliptic equation 
\begin{align*}(*)_{\varepsilon}\quad \text{div}_M\left(\dfrac{\nabla u_{\varepsilon}}{\sqrt{|\nabla u_{\varepsilon}|^2+\varepsilon^2}}\right)+\left(g^{ij}-\dfrac{\nabla^iu_{\varepsilon}\nabla^ju_{\varepsilon}}{|\nabla u_{\varepsilon}|^2+\varepsilon^2}\right)K_{ij}=\sqrt{|\nabla u_{\varepsilon}|^2+\varepsilon^2}.\quad\quad\quad\quad\quad\quad
\end{align*}
A notable feature of elliptic regularisation that is heavily exploited in this work is that the downward translating graph 
 \begin{equation}\label{translating}
\tilde{\Sigma}^{\varepsilon}_t:=\text{graph}\Bigl(\dfrac{u_{\varepsilon}}{\varepsilon}-\frac{t}{\varepsilon}\Bigr)
\end{equation}
solves the classical evolution $(*)$ in the product manifold $(M\times\mathbb{R},\bar{g}:=g\oplus dz^2)$, where we extend the given data $K$ to be parallel in the $z$-direction. Furthermore, this elliptic regularisation problem sheds new light on the study of Jang's equation (\ref{Jang}), since a rescaling of $(**)_{\varepsilon}$ can be interpreted as (\ref{Jang}) with a gradient regularisation term.

To define weak solutions to $(**)$, we use a variational principle for the energy functional
\begin{equation}\label{Ifunctional1}
\mathcal{J}^{A}_{u,\nu}(F):=|\partial^* F\cap A|-\int_{F\cap A}|\nabla u|-\left(g^{ij}-\nu^i\nu^j\right)K_{ij},
\end{equation}
defined for sets $F$ of locally finite perimeter and any compact set $A$. Here $\partial^*F$ denotes the reduced boundary of $F$, and $\nu$ represents the unit normal $\nabla u/|\nabla u|$ to the surfaces $\Sigma_t$ defined by (\ref{levelset}). The special case $K\equiv0$ corresponds to the functional employed by Huisken and Ilmanen in \cite{HI01}, and weak solutions of $(**)$ necessarily exhibit the same jumping phenomenon, characteristic of weak solutions of inverse mean curvature flow. However, since $\nabla u/|\nabla u|$ is undefined on plateaus of the locally Lipschitz function $u$, we must define an appropriate notion of normal vector in these jump regions. For this reason, a careful analysis of the jump region of the limit of the elliptic regularisation problem $(*)_{\varepsilon}$ is vital to determining the correct formulation of weak solutions to $(*)$, and constitutes a significant part of this work. On the contrary, a complete analysis of jump regions  of inverse mean curvature flow is not included in \cite{HI01}, since it was not necessary for the proof of the Riemannian Penrose Inequality. \\

The main result of this work is the following weak existence theorem.
\begin{Theorem}\label{wk}
Let $(M,g,K)$ be a complete, connected, asymptotically flat initial data set and let $E_0$ be any precompact, smooth open set in $M$. Assume that $\text{tr}_MK\geq0$ on $M\backslash E_0$. Then there exists a weak solution of $(**)$ in $M$ with initial condition $E_0$.
\end{Theorem}

The variational principle defining weak solutions also leads to a geometric characterisation of the flow, and in particular, of jump regions. We show that the level sets $\Sigma_t$ are \textit{outward optimising} (see (\ref{mini})) in the sense that they minimise ``area plus bulk energy $P$" on the outside, along the family of surfaces. This one-sided variational principle can then be exploited, via the choice of an appropriate initial condition, to prove the following existence result for MOTS.
\begin{Proposition}\label{application}
The weak solution of $(**)$ with outer trapped initial condition $E_0$ satisfying $\theta^+_{\partial E_0}<0$ will jump immediately to a smooth MOTS in $M\backslash \bar{E}_0$ at time $t=0$.\end{Proposition}

Proposition \ref{application} highlights the natural utility of this flow as a variational type approach to constructing MOTS in initial data sets containing an outer trapped surface $E_0$, where $\text{tr}_MK\geq0$ on $M\backslash E_0$.

We remark that if the mean curvature of the initial data set instead satisfies $\text{tr}_MK\leq0$, the corresponding existence result applies for the flow with speed equal to the reciprocal of $H-P$, with analogous interpretations of the solution in relation to marginally inner trapped surface (MITS) in the initial data set. \\

The results of this paper are laid out as follows. We begin in Section 2 with a brief remark on the classical evolution by inverse null mean curvature, and derive an interior estimate for the null mean curvature of smooth solutions. In Section 3 we introduce the level-set formulation of the flow, and prove existence of solutions, $u_{\varepsilon}$, to the elliptic regularisation problem. The translating graphs $\tilde{\Sigma}^{\varepsilon}_t$ given by (\ref{translating}) are then used in Section 4 to study the jump regions of the limit, $u,$ of the regularised solutions $u_{\varepsilon}$. In Section 5 we introduce the variational formulation of weak solutions, using the jump region analysis of Section 4 to motivate the choice of definition of weak solutions. In Section 6 we introduce the concept of outward optimisation to give a geometric characterisation of jump regions of weak solutions, and show that the interior of the jump region is foliated by smooth MOTS. The main existence result, Theorem \ref{wk}, then follows in Section 7, and we discuss applications of the flow, including Proposition \ref{application} in Section 8.

\vskip 0.1 true in
\textit{Acknowledgements.} I would like to sincerely thank my advisor, Gerhard Huisken, for introducing me to this topic and for his insightful and enthusiastic guidance throughout the development of this work.

\section{The smooth flow}
Since the aim of this work is to develop the weak theory for the evolution by inverse null mean curvature, we will not provide a classical PDE analysis of $(*)$, except to remark that the leading order term of the linearised equation is $\frac{1}{(H+P)^2}\Delta$ on the right hand side, where $\Delta$ denotes the Laplace-Beltrami operator with respect to the metric $g$ at time $t$. This is an elliptic operator as long as $(H+P)^{-2}$ remains non-singular, so $(*)$ is parabolic so long as the null mean curvature of the evolving surface remains strictly positive. 

In Section 3 we construct an explicit, non-compact solution $\tilde{\Sigma}^{\varepsilon}_t$ of $(*)$, for which we require an upper null mean curvature bound. The objective of this section is therefore to derive the interior $H+P$ estimate (\ref{H+P estimate}) for smooth solutions of $(*)$ (which also holds for non-compact solutions). We begin by stating the evolution equations for some fundamental quantities. Let $\nabla$ be the connection on the initial data set $(M,g,K)$ and let the induced connection and second fundamental form on $\Sigma_t$ be denoted by $D$ and $A=\{h_{ij}\}$ respectively. 
\begin{Lemma}\label{monoton}Smooth solutions of $(*)$ with $H+P>0$ satisfy the following evolution equations.\\
\noindent i)$\,\dfrac{d}{dt}H=\dfrac{1}{(H+P)^2}\Delta(H+P)-2\dfrac{|D(H+P)|^2}{(H+P)^3}-\dfrac{1}{H+P}(|A|^2+\bar{Ri}c(\nu,\nu)).\\$
ii)$\,\dfrac{d}{dt}\nu=-D\left(\frac{1}{H+P}\right) $\\
iii)$\,\dfrac{d}{dt}P=\dfrac{1}{H+P}\left(\nabla_{\nu}\text{tr}_MK-(\nabla_{\nu}K)(\nu,\nu)\right)-\dfrac{2}{(H+P)^2}D_i(H+P)K_{i\nu}.$\\
\noindent iv)$\,\dfrac{d}{dt}|\Sigma_t|+\int\limits_{V(\Sigma_t)\backslash V(\Sigma_0)}P\,dV=|\Sigma_t|$, whenever $\Sigma_0$ is closed, where $V(\Sigma)$ denotes the volume enclosed by $\Sigma$.
\end{Lemma}
\begin{proof}
The relevant evolution equations satisfied by general flows are recorded in \cite{GP99,RS00}, except for the evolution of $P$ which satisfies
\begin{tabbing}
$\,\dfrac{d}{dt}P$\=$=\dfrac{d}{dt}\text{tr}_MK-\nu^i\nu^j\dfrac{d}{dt}K_{ij}-2\nu^jK_{ij}\dfrac{d}{dt}\nu^i$\\
\>$=\dfrac{1}{H+P}\left(\nabla_{\nu}\text{tr}_MK-(\nabla_{\nu}K)(\nu,\nu)\right)-\dfrac{2}{(H+P)^2}D_i(H+P)K_{i\nu}.$
\end{tabbing}
\end{proof}
\noindent Combining \textit{i)} and \textit{ii)} of Lemma \ref{monoton} above, we obtain

\begin{equation}\label{NMCevolution}
\begin{aligned}\dfrac{d}{dt}(H+P)=&\dfrac{\Delta(H+P)}{(H+P)^2}-\dfrac{2|D(H+P)|^2}{(H+P)^3}-\dfrac{|A|^2+\bar{Ric}(\nu,\nu)}{H+P}\\
&+\dfrac{\nabla_{\nu}\text{tr}_MK-(\nabla_{\nu}K)(\nu,\nu)}{H+P}-\dfrac{2D_i(H+P)K_{i\nu}}{(H+P)^2},
\end{aligned}
\end{equation}

and for the speed function $\psi:=\dfrac{1}{H+P}$,
\begin{equation}\label{Inverseevolution}
\dfrac{\partial\psi}{\partial t}=\psi^2(\Delta\psi+(|A|^2+Ric(\nu,\nu)+\nabla_{\nu}\text{tr}_MK-(\nabla_{\nu}K)(\nu,\nu))\psi+2D_i\psi K_{i\nu}).
\end{equation}
\vskip 0.1 true in
Like in \cite{HI01}, the supremum $\sigma(x)$ of radii $r$ for which the interior curvature estimate $(\ref{H+P estimate})$ below holds is defined as follows.
\begin{Definition}\label{def1}
 Let $d_x$ denote the distance to $x$. Then for any $x\in M$, we define $\sigma(x)\in(0,\infty]$ to be the supremum of radii $R$ such that $B_R(x)\subset\subset M$, $Ric\geq-\frac{1}{100(n+1)R^2}\,\text{ in }B_R(x),$
and there exists a function $p\in C^2(B_R(x))$ such that 
\begin{equation*}
 p(x)=0,\,\, p\geq d_x^2\,\,\text{ on }\partial B_R(x),\,\text{ yet }\, |\nabla p|\leq 3d_x\,\,\text{ and }\,\nabla^2p\leq 3g\quad\text{ on }B_R(x).
\end{equation*}
\end{Definition}
\begin{Lemma}[Interior null mean curvature estimate.] \label{interior}
Let $\Sigma_t$ be a smooth solution of $(*)$ on $M$ for $0\leq s\leq t$. Then for each $x\in\Sigma_t$ and $R<\sigma(x)$
\begin{equation}\label{H+P estimate}
 H(x,t)+P(x,t)\leq \max\left((H+P)_R,\dfrac{\lambda}{R\left(\sqrt{\alpha^2+2n\lambda}-\alpha\right)}\right),
\end{equation}
where $\lambda:=4\left(3n+(12+3n)\|K\|_{C^0}R+n\|K\|_{C^1}R^2\right)$, $\alpha:=12+4n\|K\|_{C^0}R$
and $(H+P)_R$ is the maximum of $H+P$ on $\textbf{B}_{R}$, the parabolic boundary of $\Sigma_t\cap B_R(x)$.
\end{Lemma}
\begin{proof}
We wish to construct a subsolution to (\ref{Inverseevolution}).
\noindent Since
\begin{align*}
|A|^2&\geq \dfrac{H^2}{n}\geq\dfrac{1}{n}\left((H+P)^2-2P(H+P)\right),
\end{align*}
and $D\psi\leq|\nabla\psi|,$ $P\leq n\|K\|_{C^0}$ and $\nabla_{\nu}P\leq n\|K\|_{C^1}$, from  (\ref{Inverseevolution}) we obtain
\vskip 0.1 true in
\noindent\begin{align}\label{evolution1}
\dfrac{\partial\psi}{\partial t}\geq&\psi^2\Delta\psi+\dfrac{\psi}{n}-\dfrac{\psi^3}{100(n+1)R^2}-2\|K\|_{C^0}\psi^2\\
&-n\|K\|_{C^1}\psi^3-2|\nabla\psi|\,\| K\|_{C^0}\psi^2.\notag
\end{align}
\vskip 0.1 true in
\noindent We allow the evolving surface $\Sigma_t$ to have a smooth boundary $\partial\Sigma_t$, and define the parabolic boundary of the flow $\Sigma_t\cap B_R$ to be 
\begin{equation*}
\textbf{B}_{R}=\textbf{B}_{R}(x,t):=(B_R\cap\Sigma_0)\times\{0\}\cup(\cup_{0\leq s\leq t}(B_R\cap\partial\Sigma_s)\times\{s\}), 
\end{equation*}
and 
\begin{equation*}(H+P)_R=(H+P)_R(x,t):=\sup_{(y,s)\in \textbf{B}_R}H(y,s)+P(y,s).
\end{equation*}

\noindent Consider the function $\phi=\phi_{\delta}(y):=\dfrac{C_{\delta}}{R}\bigl(R^2-p(y)\bigr),$ where 
\begin{equation*}C_{\delta}:=\bigl(\max\bigl(R(H+P)_R,\,\frac{\lambda}{\left(\sqrt{\alpha^2+2n\lambda}-\alpha\right)}\bigr)\bigr)^{-1}-\delta,\end{equation*} 
for $0<\delta\ll1$ and $p$ as defined above. Note that $\Delta\phi=\text{tr}_{\Sigma_t}(\nabla^2\phi)-H\langle\nabla\phi,\nu\rangle.$
Then for $y\in \Sigma_t\cap B_R$, we have
\begin{align}
 \left(\dfrac{\partial}{\partial t}-\phi^2\Delta\right)\phi&=\langle\nabla\phi,\dfrac{\partial y}{\partial t}\rangle-\phi^2\text{tr}_{\Sigma_t}\nabla^2\phi+H\phi^2\langle\nabla\phi,\nu\rangle\notag\\
&=-\dfrac{C_{\delta}}{R}\langle\nabla p,\nu\rangle\left(\psi+\dfrac{\phi^2}{\psi}-P\phi^2\right)+\phi^2\dfrac{C_{\delta}}{R}\text{tr}_{\Sigma_t\cap B_R}(\nabla^2p).\label{evolution2}
\end{align}
Since $\phi\leq C_{\delta}R\leq\dfrac{1}{(H+P)_R}-\delta R<\psi$, it follows that $\phi<\psi$ on $\textbf{B}_R$. In order to obtain a contradiction, let $0<s\leq t$ denote the first time when $(\psi-\phi)(y,s)=0$ for $y\in\Sigma_s\cap B_R(x)$. At this point
\begin{align*}
 \left(\dfrac{\partial}{\partial t}-\phi^2\Delta\right)(\psi-\phi)\leq 0.
\end{align*}
On the other hand, since $\phi<R$, it follows from (\ref{evolution1}),  (\ref{evolution2}) and the conditions on $p$ defined above that at the point $(y,s)$
\begin{align*}
 \Biggl(&\dfrac{\partial}{\partial t}-\phi^2\Delta\Biggr)(\psi-\phi)>\\
 >&\phi\biggl(\dfrac{1}{2n}-2\|K\|_{C^0}\phi-n\|K\|_{C^1}\phi^2-2|\nabla\phi|\,\| K\|_{C^0}\phi+\frac{C_{\delta}}{R}\langle\nabla p,\nu\rangle(2-P\phi)\\
 &\quad-\phi\dfrac{C_{\delta}}{R}\text{tr}_{\Sigma_t\cap B_R}(\nabla^2p)\biggr)\\
\geq& \phi\biggl(\dfrac{1}{2n}-2C_{\delta}\Bigl(3+n\|K\|_{C^0}R\Bigr)-C_{\delta}^2\bigl(3n+12\| K\|_{C^0}R+n\|K\|_{C^1}R^2\\
&\quad+3n\|K\|_{C^0}R\bigr)\biggr)=0.
\end{align*}
Thus $\psi>\phi$ on all of $\Sigma_t\cap B_R(x)$. In particular $\psi(x,t)>\phi(x,t)=C_{\delta}R,$ and as $\delta$ was arbitrary it follows that $\psi(x,t)\geq C_0R$. 
\end{proof}

In Section \ref{3.2} we see that the null mean curvature upper bound given by Lemma \ref{interior} is the key to existence and regularity, and that this estimate continues to hold for weak solutions. On the other hand, the reaction term $-\frac{|A|^2}{H+P}$ in the evolution (\ref{NMCevolution}) of the null mean curvature in general leads to singularity formation in finite time, analogous to inverse mean curvature flow. We therefore turn to the question of a weak formulation of solutions to the evolution by inverse null mean curvature.
\section{Level-set description and elliptic regularisation}\label{3.2}
In this section we outline a level-set description of the evolution by inverse null mean curvature. This level set formulation allows jumps in a natural way, because if $u$ is constant on an open set $\Omega$, the level sets ``jump" across $\Omega$. We use the method of elliptic regularisation as a tool to approximate solutions of the degenerate elliptic level set problem by smooth solutions of a strictly elliptic equation. Studying the properties of the regularised solutions helps to guide us towards the optimal formulation for weak solutions of $(**)$, which we then define in Section \ref{2.3}. \\

\textbf{Level-Set Formulation.} The following ansatz lies at the foundation of the level-set formulation. We assume that the evolving surfaces are given by the level-sets of a scalar function $u:M\to\mathbb{R}$ via 
\begin{equation}\label{time of arrival}
E_t:=\{x:u(x)<t\},\quad\quad\Sigma_t:=\partial E_t.
\end{equation}
Employing the terminology coined by Brian White in  \cite{W11}, we call $u$ the \textit{time-of-arrival function} for the evolution by null mean curvature. Then wherever $u$ is smooth and $\nabla u\neq0$, the normal vector to $\Sigma_t$ is given by $\nu=\dfrac{\nabla u}{|\nabla u|}$ and the degenerate elliptic boundary value problem\\
\begin{equation}\tag{$**$} \,\,
\left\lbrace
\begin{aligned}
 \mbox{div}\left( \frac{\nabla u}{|\nabla u|}\right)+\left(g^{ij}-\dfrac{\nabla^iu\nabla^ju}{|\nabla u|^2}\right)K_{ij} &= |\nabla u| , \\
           u\Big|_{\partial E_0} \,\, &= \,\, 0,
\end{aligned}
\right.
\end{equation}
describes the evolution of the level-sets of $u$ by inverse null mean curvature. In this smooth setting, the left hand side is the null mean curvature of $\Sigma_t$ and the right hand side is the inverse speed of the family of level-sets. Since $|\nabla u|=H+P$, the local uniform estimate (\ref{H+P estimate}) for the null mean curvature suggests that it is reasonable to expect locally Lipschitz solutions of $(**)$. However, in order to interpret $(**)$ as the level-set formulation of the classical, expanding flow $(*)$, it is necessary for the zero function be a subsolution barrier for the Dirichlet problem $(**)$. In particular, this suggests that it only makes sense to study $(**)$ on initial data sets $(M,g,K)$ satisfying $\text{tr}_MK\geq0$ on $M\backslash E_0$. We see below that this mean curvature restriction is also necessary for the elliptic regularisation problem.\\

\textbf{Elliptic regularisation.}
In order to solve the degenerate elliptic problem $(**)$, we study solutions of the following strictly elliptic equation on the precompact domain $\Omega_L:=F_L\backslash\bar{E}_0$,                                                       
\small\begin{align*}(*)_{\varepsilon}
\begin{cases}\mbox{div}\left(\dfrac{\nabla u_{\varepsilon}}{\sqrt{|\nabla u_{\varepsilon}|^2+\varepsilon^2}}\right)+\left(g^{ij}-\dfrac{\nabla^iu_{\varepsilon}\nabla^ju_{\varepsilon}}{|\nabla u_{\varepsilon}|^2+\varepsilon^2}\right)K_{ij}=\sqrt{|\nabla u_{\varepsilon}|^2+\varepsilon^2}&\text{in }\Omega_L,\\
u_{\varepsilon}=0&\text{on }\partial E_0,\\
u_{\varepsilon}=L-2&\text{on }\partial F_L,
\end{cases}
\end{align*}\normalsize
\noindent where $F_L:=\{v<L\}$ for an appropriate comparison function $v$ defined below. In this section we prove existence of a smooth solution of $(*)_{\varepsilon}$. \\

Rescaling $(*)_{\varepsilon}$ via $\hat{u}_{\varepsilon}:=\dfrac{u_{\varepsilon}}{\varepsilon }$ gives
\begin{equation*}(*)_{\hat{\varepsilon}} \quad\quad
\mbox{div}\left(\dfrac{\nabla \hat{u}_{\varepsilon}}{\sqrt{|\nabla \hat{u}_{\varepsilon}|^2+1}}\right)+\left(g^{ij}-\dfrac{\nabla^i\hat{u}_{\varepsilon}\nabla^j\hat{u}_{\varepsilon}}{|\hat{u}_{\varepsilon}|^2+1}\right)K_{ij}=\varepsilon\sqrt{|\nabla \hat{u}_{\varepsilon}|^2+1},\quad\quad\quad\quad\quad
\end{equation*}
and we see that the left hand side corresponds to the null mean curvature $\hat{H}_{\varepsilon}+\hat{P}_{\varepsilon}$ of the hypersurface $\text{graph}(\hat{u}_{\varepsilon})$ in the product manifold $(M^{n+1}\times\mathbb{R},\bar{g}),$ where $\bar{g}:=g\oplus dz^2$ and the given data $K$ is extended to be constant in the $z$-direction. This rescaled equation $(\ast)_{\hat{\varepsilon}}$ has the geometric interpretation that the downward translating graph
\begin{equation} \label{graph}
 \tilde{\Sigma}^{\varepsilon}_t:=\mbox{graph }\left(\hat{u}_{\varepsilon}-\dfrac{t}{\varepsilon}\right),
\end{equation}
solves $(*)$ smoothly in $\Omega_L\times\mathbb{R}$. This is equivalent to the statement that the function 
\begin{equation}\label{levelsetepsilon}U_{\varepsilon}(x,z):=u_{\varepsilon}(x)-\varepsilon z,\quad\quad(x,z)\in\Omega_L\times\mathbb{R},
\end{equation}
solves $(**)$ in $\Omega_L\times\mathbb{R}$, since $\tilde{\Sigma}^{\varepsilon}_t=\{U_{\varepsilon}=t\}$, that is, $U_{\varepsilon}$ is the time-of-arrival function for the solution $\tilde{\Sigma}^{\varepsilon}_t$.
Therefore elliptic regularisation allows one to approximate solutions of $(**)$ by smooth, noncompact solutions of $(*)$ one dimension higher. \\

In fact, $(*)_{\hat{\varepsilon}}$ has the further interpretation as Jang's equation (\ref{Jang}) with a gradient regularisation term. In \cite{J78}, Jang used equation (\ref{Jang}) to generalise Geroch's \cite{G73} approach to proving the positive mass theorem from the time symmetric case to the general case. He noted, however, that the equation cannot be solved in general, leaving the question of existence and regularity of solutions open. The analytical difficulty is the lack of an a priori estimate for $\sup |w|$ due to the presence of the zero order term $tr_M(K)$. In \cite{SY81}, Schoen and Yau bypass this issue using a positive capillarity regularisation term which provides a direct sup estimate via the maximum principle. In the case of the Dirichlet problem for Jang's equation, appropriate trapping assumptions must be placed on the boundaries in order to obtain the required boundary gradient estimates (see \cite{AM09} and \cite{E09}, or \cite{AEM09} for an overview).\\

In the case of the Dirichlet problem $(*)_{\varepsilon}$, we see below that the zero order term $tr_M(K)$ obstructs the existence of a subsolution barrier at the inner boundary. In order to obtain the required boundary gradient estimate at this inner boundary, we must impose the ambient mean curvature restriction mentioned previously, that $\text{tr}_M(K)=g^{ij}K_{ij}$ is nonnegative on $M\backslash E_0$.  Similarly, it was observed by J. Metzger \cite{M10} that restricting to $\text{tr}_MK\geq 0$ in the capillarity regularised problem prevents the solution from blowing-up to negative infinity over marginally inner trapped surfaces in the initial data set. 

\textbf{A priori estimates and existence for $(*)_{\varepsilon}$.}
As stated above, we will use a comparison function $v$ to prescribe the outer boundary $\partial F_L$ of the annulus domain $\Omega_L$ for the Dirichlet problem $(*)_{\varepsilon}$. Since $M$ is asymptotically flat, outside some compact set $\Omega\subset M$ we can choose a radial coordinate chart such that for an appropriately chosen $\alpha>0$, the function $v=\alpha\log r$ is a smooth subsolution of the following approximating level-set equation
\begin{equation}\label{modified}
\text{div}\left(\dfrac{\nabla u}{|\nabla u|}\right)+s\left(g^{ij}-\dfrac{\nabla^iu\nabla^ju}{|\nabla u|^2}\right)K_{ij}=|\nabla u|,
\end{equation}
for $s\in[0,1]$ in this asymptotic region $M\backslash \Omega$. Let 
\small\begin{equation*}
E^{\varepsilon,s}u_{\varepsilon,s}:=\mbox{div}\left(\dfrac{\nabla u_{\varepsilon,s}}{\sqrt{|\nabla u_{\varepsilon,s}|^2+\varepsilon^2}}\right)+s\left(g^{ij}-\dfrac{\nabla^iu_{\varepsilon,s}\nabla^ju_{\varepsilon,s}}{|\nabla u_{\varepsilon,s}|^2+\varepsilon^2}\right)K_{ij}-\sqrt{|\nabla u_{\varepsilon,s}|^2+\varepsilon^2},
\end{equation*}\normalsize
To prove existence of solutions to the Dirichlet problem $(*)_{\varepsilon}$, we then consider solutions of the family of approximating equations
\begin{align*}(*)_{\varepsilon,s}
\begin{cases} E^{\varepsilon,s}u_{\varepsilon,s}=0&\text{in }\Omega_L,\\
u_{\varepsilon,s}=0&\text{on }\partial E_0,\\
u_{\varepsilon,s}=s(L-2)&\text{on }\partial F_L,\end{cases}
\end{align*}
for $s\in[0,1]$, where the subsolution $v=\alpha\log r$ prescribes the outer boundary $\partial F_L=\partial\{v<L\}$ for both the Dirichlet problems $(*)_{\varepsilon,s}$ and $(*)_{\varepsilon}$. We use barrier functions at the inner and outer boundaries to derive the following interior and boundary gradient estimates. Aside from the supersolution barrier at the outer boundary, the following Lemma follows essentially as in \cite[Lemma 3.4]{HI01}.

\begin{Lemma}\label{apriori}For every $L>0$, there exists $\varepsilon(L)>0$ such that for $0<\varepsilon<\varepsilon(L)$ and $s\in[0,1]$, a smooth solution of $(*)_{\varepsilon,s}$ on $\bar{\Omega}_L$ satisfies the following a priori estimates:
\begin{equation}\label{supest1}
u_{\varepsilon,s}\geq -\varepsilon\quad\text{in }\bar{\Omega}_L,\quad u_{\varepsilon,s}\geq v+(s-1)(L-2)-2\quad\text{in }\bar{F}_L\backslash F_0,
\end{equation}
\begin{equation}
u_{\varepsilon}\leq C(L,\|K\|_{C_0})\quad\text{in }\bar{\Omega}_L,
\end{equation}
\begin{equation}\label{bndryest}
|\nabla u_{\varepsilon,s}|\leq H_+\varepsilon+n|p|\,\,\,\text{on }\partial E_0,\quad\quad |\nabla u_{\varepsilon,s}|\leq C(L,\|K\|_{C^0})\,\,\,\text{on }\partial F_L,
\end{equation}
\begin{equation}\label{intgrad}
|\nabla u_{\varepsilon,s}(x)|\leq\max_{\partial\Omega_L\cap B_r(x)}|\nabla u_{\varepsilon,x}|+\varepsilon+C,\quad x\in\bar{\Omega}_L,
\end{equation}
\begin{equation}
|u_{\varepsilon,s}|_{C^{2,\alpha}(\bar{\Omega}_L)}\leq C(\varepsilon, L,n,\|K\|_{C^0},\|K\|_{C^1}).
\end{equation}
\end{Lemma}
\begin{proof}
Let $|\lambda|$ denote the size of the largest eigenvalue of $K_{ij}$ on $\bar{\Omega}_L$.\\
\noindent 1a. We construct a subsolution that bridges from $E_0$ to where $v$ starts in the asymptotic region, which allows for unrestricted jumps in the compact part of the manifold.\\

Let $\text{Cut}(E_0)$ be the cut locus of $E_0$ in $M$. We construct a subsolution to $(*)_{\varepsilon}$ on $M\backslash \bigl(E_0 \cup\text{Cut}(E_0)\bigr)$. Define $G_0=E_0,\,G_d:=\{x:dist(x,E_0)<d\}$ and choose $d_L$ large enough that $G_{d_L}\supseteq F_L$. In general, for a surface moving in normal direction with speed $f$, the evolution of the mean curvature is given by
\begin{equation}\label{geodesic}
 \dfrac{\partial H}{\partial t}=-\Delta f-|A|^2f-Ric(\nu,\nu)f.
\end{equation}
We can therefore estimate the mean curvature of the surfaces $\partial G_d$ via
\begin{equation*}
 \dfrac{\partial H}{\partial d}=-|A|^2-Ric(\nu,\nu)\leq C_1(L)\quad\quad\text{on }\partial G_d\backslash\text{Cut}(E_0),\,\,0\leq d\leq d_L,
\end{equation*}
yielding
\begin{equation*}
 H_{\partial G_d}\leq \max_{\partial E_0}H_++C_1d\leq C_2(L)\quad\quad\text{on }\partial G_d\backslash\text{Cut}(E_0),\,0\leq d\leq d_L,\quad\quad\,\,
\end{equation*}
where $H_+=\max(0,H)$.
Consider the prospective subsolution 
\begin{equation*}
 v_1(x):=f(d)=f(\text{dist}(x,G_0)),\quad x\in\bar{G}_{d_L}\backslash E_0,\, f'<0.
\end{equation*}
Since $\nabla v_1=f'\nu$, we have $\nabla^2_{ij}v_1=f'\langle\nabla_{e_i}\nu,e_j\rangle=f'h_{ij}$ and thus
\begin{align*}(g^{ij}-\nu^i\nu^j)\nabla^2_{ij}v_1&=f' H_{\partial G_d}\geq f'C_2.
\end{align*}

\noindent Hence
\begin{small}\begin{align*}
 \sqrt{f'^2+\varepsilon^2} E^{\varepsilon,s} v_1&\geq -|f'|C_2+\dfrac{\varepsilon^2f''}{f'^2+\varepsilon^2}+s\sqrt{f'^2+\varepsilon^2}g^{ij}K_{ij}-|f'||K_{\nu\nu}|-f'^2-\varepsilon^2.
\end{align*}\end{small}
In order to obtain an appropriate subsolution, it is in fact necessary to restrict to initial data sets $(M,g,K)$ with $g^{ij}K_{ij}\geq0$, so that we can discard the bad term $s \sqrt{f'^2+\varepsilon^2}g^{ij}K_{ij}$. Then we can use the following barrier
\begin{equation*}f(d):=\dfrac{\varepsilon}{A}(-1+e^{-Ad})\quad\text{ on  }0\leq d\leq d_L.
\end{equation*}
\noindent If we restrict $\varepsilon$ such that  $\varepsilon\leq e^{-Ad_L}$, then $|f'|=\varepsilon e^{-Ad}\geq \varepsilon^2$ and $\varepsilon^2\leq|f'|\leq\varepsilon$. Then taking $A:=2(C_2+|\lambda|+2)$ we obtain
\begin{align*}
 (f'^2+\varepsilon^2)(|f'|C_2+ |f'||K_{\nu\nu}|+f'^2+\varepsilon^2)&\leq 2\varepsilon^2(C_2+|K_{\nu\nu}|+2)|f'|\leq \varepsilon^2 f''.
\end{align*}
This shows that the function $v_{1,s}(x):=\dfrac{\varepsilon}{2(C_2+|\lambda|+2)}\left(e^{-(2C_2+|\lambda|+2)d}-1\right)$ is a smooth subsolution for $E^{\varepsilon,s}$ on $G_{d_L}\backslash \bigl(E_0\cup\text{Cut}(E_0)\bigr)$ for sufficiently small $\varepsilon$. Furthermore, $v_{1,s}$ is a viscosity subsolution of $E^{\varepsilon,s}$ on all of $G_{d_L}\backslash\bar{E}_0$. Since $u\geq v_1$ on the boundary, it follows by the maximum principle for viscosity solutions \cite[Thm 3.3]{CIL92} that 
\begin{equation}\label{eep}
 u\geq v_1\geq -\varepsilon\quad\,\text{in }\bar{\Omega}_L,\,\text{ and   }\,\,\,\,\dfrac{\partial u}{\partial \nu}\geq\dfrac{\partial v_1}{\partial \nu}\geq-\varepsilon\quad\text{on }\partial E_0.
\end{equation}

\noindent b) We construct a subsolution on $\bar{F}_L\backslash F_0$. Assume $L>1$ and consider the function 
$v_2:=\dfrac{L-1}{L}v-1+(s-1)(L-2).$ Then $E^{0,s}v_2=E^{0,s}v+\dfrac{1}{L}|\nabla v|>0$ on $\bar{F}_L\backslash F_0$. Since the domain is compact, for all sufficiently small $\varepsilon$ we obtain $E^{\varepsilon,s}v_2>0$. From (\ref{eep}) we have that $u\geq -\varepsilon$ in $\bar{\Omega}_L$, thus 
$u\geq v_2$ on $\partial F_0$ and $u=s(L-2)=v_2$ on $\partial F_L$. It then follows from the maximum principle that $u\geq v_2\geq v+(s-1)(L-2)-2\quad\text{in }\bar{F}_L\backslash F_0,$ thus
\begin{equation}
 \dfrac{\partial u}{\partial \nu}\geq -C(L)\quad\quad\text{on }\partial F_L.
\end{equation}
A rescaled version of $v_2$ provides the required barrier when $L\leq 1$. 
\vskip 0.1 true in

\noindent c) The zero order term $tr_M(K)$ prevents constant functions from being supersolutions to $(*)_{\varepsilon,s}$, like in inverse mean curvature flow. We therefore construct a linear supersolution to $(*)_{\varepsilon}$ on $F_L\backslash \bigl(E_0\cup\text{Cut}(\partial F_L)\bigr)$, where $\text{Cut}(\partial F_L)$ is the cut locus of $\partial F_L$ in $\bar{F}_L$. Consider $v_3(x):=f(d)=f(\text{dist}(x,G_0))$ where $G_0:=\partial F_L$, $G_d:=\{x:dist(x,\partial F_L)<d\}$ and choose $d_0$ large enough that $G_{d_0}\supseteq \partial E_0$. From $(\ref{geodesic})$ we find $\dfrac{\partial H}{\partial d}\geq -C_1(L)$ on $\partial G_d\backslash\text{Cut}(\partial F_L)$, for $0\leq d\leq d_0,$ yielding
\begin{equation*}
 H_{\partial G_d}\geq -\max_{\partial F_L}H_--C_1d\geq -C_2(L)\quad\quad\text{on }\partial G_d\backslash\text{Cut}(\partial F_L),\,\,0\leq d\leq d_0,
\end{equation*}
where $H_-=-\min(H,0)$.
\noindent Setting $v_3(x)=f(d):=s(L-2)+\left(m+\dfrac{2}{d_0}\right)d,$ where $m>0$ is to be chosen, we obtain
\begin{align*}
\sqrt{f'^2+\varepsilon} E^{\varepsilon,s} f(d)&\leq f'(C_2+2|g^{ij}K_{ij}|+|K_{\nu\nu}|-f').
      \end{align*}
Setting $m:=C_2+2|g^{ij}K_{ij}|+|\lambda|$ ensures $\sqrt{f'^2+\varepsilon} E^{\varepsilon,s} f(d)\leq0$ for all sufficiently small $\varepsilon$ (so that $\varepsilon\leq f'$). Then $v_3(x)$ is a smooth supersolution on $G_{d_0}\backslash \bigl(E_0\cup \text{Cut}(\partial F_L)\bigr)$. Furthermore, $v_3$ is a viscosity subsolution on all of $G_{d_0}\backslash\bar{E}_0$. Since $u=f$ on $\partial F_L$ and $u<f$ on $\partial E_0$, it follows by the maximum principle for viscosity solutions that 
\begin{align}\label{supestimate}
 u\leq f\leq L+m\,d_0\quad\quad\quad\quad\quad\quad\,\,&\text{in }\bar{\Omega}_L,\\
\dfrac{\partial u}{\partial \nu}\leq C_2+2|g^{ij}K_{ij}|+|p|+\dfrac{2}{d_0}=C(L,\|K\|_{C^0})\quad\quad\,\,\,\,\,\,\,&\text{on }\partial F_L.
\end{align}
\vskip 0.1 true in
\noindent d) Choose a smooth function $v_4$ which vanishes on $\partial E_0$ such that 
\begin{equation}
 H_++n\|K\|_{C^0}<\dfrac{\partial v_4}{\partial\nu}\leq H_++\varepsilon+n\|K\|_{C^0}\quad\quad\text{along }\partial E_0.
\end{equation}
Let $\nu$ be the normal vector to $\partial E_0$, and $\tau$ be the tangent to $\partial E_0$, which satisfies 
\begin{equation*}\nu=\lambda\dfrac{\nabla v_4}{|\nabla v_4|}+\sqrt{\lambda^2-1}\tau,\quad \text{for some }\lambda\geq 1.\end{equation*}
Then
\begin{align*}
 \text{div}\left(\dfrac{\nabla v_4}{|\nabla v_4|}\right)&=\dfrac{1}{\lambda}\text{div}\left(\dfrac{\nabla u}{|\nabla u|}\right)-\dfrac{\sqrt{\lambda^2-1}}{\lambda}\text{div}(\tau)=\dfrac{1}{\lambda}H_{\partial E_0},
\end{align*}
thus along $\partial E_0$ we obtain
\begin{equation*}
 E^{0,s}v_4<\frac{1}{\lambda}H_{\partial E_0}+n\|K\|_{C^0}-\bigl(H_++n\|K\|_{C^0}\bigr)\leq0.
\end{equation*}
This implies that $E^0v_4<0$ in the neighbourhood $U:=\{0\leq v_4\leq \delta\},$ for sufficiently small $\delta>0$. Now define the scaled-up function $v_5:=\dfrac{v_4}{1-v_4/\delta},$ for $x\in U$. So $v_5\rightarrow\infty$ for $v_4\rightarrow\delta$, ie on $\partial U\backslash\partial E_0$, and $ E^{0,s}v_5=E^{0,s}v_4+|\nabla v_4|-\dfrac{|\nabla v_4|}{(1-v_4/\delta)^2}\leq E^{0,s}v_4<0.$ 

For $\varepsilon$ sufficiently small (depending on $L$ and $m$) we obtain that $E^{\varepsilon,s}v_5<0$ on the set $V:=\{0\leq v_5\leq L+m\,d_0\}$. From $(\ref{supestimate})$ we have $u\leq L+m\,d_0$ on $\bar{\Omega}_L$, thus $u\leq v_5 $ on $\partial V$.\\
Then by the maximum principle, $u\leq v_5$ on $V$, and therefore\\
\begin{equation}\label{E_0}
\dfrac{\partial u}{\partial\nu}\leq\dfrac{\partial v_5}{\partial\nu}=\dfrac{\partial v_4}{\partial\nu}\leq H_++\varepsilon+n|\lambda|\quad\quad\quad \text{on }\partial E_0
\end{equation}
for sufficiently small $\varepsilon$.\\
\vskip 0.1 true in
\noindent 2) The desired interior gradient estimate (\ref{intgrad}) can be obtained from the interior estimate for $H+P$ in Lemma \ref{interior}. Since we can not apply the result directly to $(*)_{\varepsilon,s}$ (except when $s=1$), we instead rework the proof of Lemma \ref{interior} for the evolution equation
\begin{flushleft}$(*)_s\quad\quad\quad\quad\quad\quad\quad\quad\quad\quad\quad\dfrac{\partial F}{\partial t}=\dfrac{1}{H+sP}\nu,\quad\quad s\in[0,1],$\\
 \end{flushleft}
to obtain the corresponding estimate 
\begin{equation}\label{s}
H(x,t)+sP(x,t)\leq \max\left((H+sP)_R,\frac{\lambda}{\left(\sqrt{\alpha^2+2n\lambda}-\alpha\right)}\right).
\end{equation}
Here $\lambda$ and $\alpha$ are as defined above, and $(H+sP)_R$ is the maximum of $H+sP$ on $P_R$, the parabolic boundary of $\Sigma^{^{\Large{s}}}_t\cap B_R(x)$.\\
\\
Analogous to (\ref{graph}), the downward translating graph $\tilde{\Sigma}^{\varepsilon,s}_t:=\text{graph}\left(\dfrac{u_{\varepsilon,s}}{\varepsilon}-\dfrac{t}{\varepsilon}\right)$
is a smooth solution of $(*)_s$, described by the level-set function 
$U_{\varepsilon,s}(x,z):=u_{\varepsilon,s}(x)-\varepsilon z,$
since $\tilde{\Sigma}^{\varepsilon,s}_t=\{U_{\varepsilon,s}=t\}$. We then relate estimate (\ref{s}) to $|\nabla u_{\varepsilon,s}|$ via $(*)_{\varepsilon,s}$, which asserts that
\begin{equation*}
 (H+sP)_{\tilde{\Sigma}^{\varepsilon,s}_t}=\sqrt{|\nabla u_{\varepsilon,s}|^2+\varepsilon^2}.
\end{equation*}
Now let $\textbf{B}:=B^{n+1}_R(x,z),$ be an (n+1)-dimensional ball centered at the point $(x,z)\subset M\times\mathbb{R}$. Since $\tilde{\Sigma}^{\varepsilon,s}_t$ is a translating solution to $(*_s)$, its parabolic boundary is just a translation of $\partial\Omega_L$ in time. Furthermore, as $|\nabla u_{\varepsilon,s}|$ is independent of $z$, applying (\ref{s}) to $\tilde{\Sigma}^{\varepsilon,s}_t\cap\textbf{B}$ yields
\small\begin{align*}
\sqrt{|\nabla u_{\varepsilon,s}|^2+\varepsilon^2}&\leq\sup_t\max_{\partial\tilde{\Sigma}^{\varepsilon,s}_t\cap\textbf{B}}\sqrt{|\nabla u_{\varepsilon,s}|^2+\varepsilon^2}+C\leq\max_{\partial\Omega_L\cap B^n_R(x)}|\nabla u_{\varepsilon,s}|+\varepsilon+C,
\end{align*}\normalsize
where $C:=\frac{\lambda}{\left(\sqrt{\alpha^2+2n\lambda}-\alpha\right)}$ is defined in Lemma \ref{interior}. For $\varepsilon$ small enough, we obtain from the boundary gradient estimates 
\begin{equation}\label{lipschitz}
 |\nabla u_{\varepsilon,s}(x)|\leq\max_{\partial E_0\cap B_R(x)}H_++2+\tilde{C}(L,\|K\|_{C^0})+C,\\
\end{equation}
which leads to the Lipschitz estimate
\begin{equation*}
 |u_{\varepsilon,s}|_{C^{0,1}(\bar{\Omega}_L)}\leq C(L,n,\|K\|_{C^1}).
\end{equation*}
Then by reworking the proof of the Nash-Moser-De Giorgi estimate (\cite{GT01}, Thm 13.2), we obtain
\begin{equation*}
 |u_{\varepsilon,s}|_{C^{1,\alpha}(\bar{\Omega}_L)}\leq C(\varepsilon,L,n,\|K\|_{C^1}),
\end{equation*}
for some $\alpha=\alpha(\Omega_L)$. This implies a bound on the H\"older modulus of continuity for the coefficients of $E^{\varepsilon,s}u$, so Schauder theory improves this estimate to $C^{2,\alpha}$
\begin{equation}\label{C2alpha}
 |u_{\varepsilon,s}|_{C^{2,\alpha}(\bar{\Omega}_L)}\leq C(\varepsilon,L,n,\|K\|_{C^1}).
\end{equation}
\end{proof}
\begin{Lemma}[Existence for the regularised problem] \label{approxlemma}
A smooth solution of $(*)_{\varepsilon}$ exists.
\end{Lemma}
\begin{proof} We first prove there is a solution of $(*_{\varepsilon,s})$ for $s=0$ and small $\varepsilon$. Let $\hat{u}=\frac{u_{\varepsilon}}{\varepsilon}$ and rewrite $(*)_{\varepsilon,0}$ as
\begin{align*}(*)_{\hat{\varepsilon}}\begin{cases}F(\hat{u}):=\dfrac{1}{\sqrt{|\nabla \hat{u}|^2+1}}\,\text{div}\left(\dfrac{\nabla \hat{u}}{\sqrt{|\nabla \hat{u}|^2+1}}\right)=\varepsilon\,\,\quad\quad&\text{in }\Omega_L,\\
                  \hat{u}=0\quad\quad\quad\quad\quad\quad\quad\quad\quad\quad\quad\quad\quad\quad\quad\quad\quad\quad\quad\quad\,&\text{on }\partial\Omega_L.
                 \end{cases}
\end{align*}
The map 
\begin{equation*}
 F:C_0^{2,\alpha}(\bar{\Omega}_L)\rightarrow C^{\alpha}(\bar{\Omega}_L),
\end{equation*}
is $C^{1}$, and possesses the solution $F(0)=0$ for $\varepsilon=0$. The linearisation of $F$ at $\hat{u}=0$ is
\begin{equation*}
 \mathcal{D}F\vert_0=\Delta_{g}:C^{2,\alpha}_0(\bar{\Omega}_L)\rightarrow C^{\alpha}(\bar{\Omega}_L).
\end{equation*}
The Laplacian on $M$ is an isomorphism, so by the Implicit Function Theorem there exists $\varepsilon_0>0$ such that $(*)_{\hat{\varepsilon}}$ has a unique solution for $0\leq\varepsilon<\varepsilon_0$.
\vskip 0.2 true in
\noindent We now fix $\varepsilon\in(0,\varepsilon_0)$ and vary $s$. Let $I$ be the set of $s$ such that $(*)_{\varepsilon,s}$ has a solution $u_{\varepsilon,s}\in C^{2,\alpha}(\bar{\Omega}_L).$ We have shown that $I$ contains $0$. We first show that $I$ is open. Let $\pi$ be the boundary value map $u\mapsto u\vert_{\partial\Omega}$. Consider the map 
\begin{equation*}
 G:C^{2,\alpha}(\bar{\Omega}_L)\times\mathbb{R}\rightarrow C^{\alpha}(\bar{\Omega}_L)\times C^{2,\alpha}(\partial\Omega_L),
\end{equation*}
defined by $G(w,s):=G^{s}(w)=\biggl(E^{\varepsilon,s}(w),\pi(w)-s(L-2)\chi_{\partial F_L}\biggr)$, so that $(*)_{\varepsilon,s}$ is equivalent to $G^s(w)=(0,0)$. $G^s(w)$ is $C^{1}$, and possesses the solution $G^0(u_0)=(0,0)$, where  
 $u_0$ is the $C^{2,\alpha}_0(\bar{\Omega}_L)$ solution from above. The linearisation of $G^0$ at $u_0$ is the operator $\mathcal{D}G^0_{u_0}$ given by
\begin{equation}
 \mathcal{D}G^0_{u_0}=\vs A^{ij}\nabla_i\nabla_j+B^i\nabla_i\\ \pi\ve :C^{2,\alpha}\rightarrow C^{\alpha}(\bar{\Omega}_L)\times C^{2,\alpha}(\partial\Omega_L),
\end{equation}
where 
\begin{align*}
 A^{ij}=\dfrac{1}{\sqrt{1+|\nabla u_0|^2}}\left(g^{ij}-\dfrac{\nabla^iu_0\nabla^ju_0}{1+|\nabla u_0|^2}\right),\quad B^i=\nabla_jA^{ij}-\varepsilon\dfrac{\nabla^iu_0}{\sqrt{|\nabla u_0|^2+1}}.
\end{align*}
Since $\mathcal{D}E^0_{u_0}(w)$ is a linear elliptic equation with H\"older continuous coefficients, we can apply Schauder theory (eg \cite{GT01}, Thm 6.14) to deduce that $\mathcal{D}G^0_{u_0}$ is a bijective map with continuous inverse. It follows from the Implicit Function Theorem that $G$ maps a neighbourhood of $(u_0,0)$ onto a neighbourhood of $(0,0)$. Thus $I$ is relatively open, which completes the proof of existence of $u^{\varepsilon}\in C^{2,\alpha}(\bar{\Omega})$ solving $(*)_{\varepsilon}$. Smoothness then follows from standard Schauder estimates.
\end{proof}
In view of the local uniform Lipschitz estimates for $u_{\varepsilon}$, by the Arzela Ascoli theorem there exist sequences $\varepsilon_i\to0$, $L_i\to\infty$, a subsequence $u_i$ and a locally Lipschitz function $u:M\backslash E_0\to\mathbb{R}$ such that 
\begin{equation}\label{subseq}
u_{i}\to u 
\end{equation}
locally uniformly on $M\backslash E_0$, and by (\ref{lipschitz}), $u$ satisfies
\begin{equation}
|\nabla u(x)|\leq \sup_{\partial E_0\cap B_R(x)}H_++C(L,n,\|K\|_{C^1},R).
\end{equation} 
In the next section we will study the limit of the translating graphs $\tilde{\Sigma}^{\varepsilon}_t=\{U_{\varepsilon}=t\}$, where the time-of-arrival function $U_{\varepsilon}$ was defined by (\ref{levelsetepsilon}). By setting 
\begin{equation}\label{Udefn1}
U(x,z):=u(x),
\end{equation}
we obtain that $U_{i}\to U$ locally uniformly on $(M\backslash E_0)\times\mathbb{R}$, therefore $U$ is the time of arrival function of the limit of the smooth flow $t\mapsto\tilde{\Sigma}^{i}_t$.\\
\section{The limit of the translating $\varepsilon$-graphs $\tilde{\Sigma}^{\varepsilon}_t$.}\label{sec:jump}
Our choice of variational formulation to define weak solutions to $(**)$, detailed in the next section, is motivated by:
\begin{enumerate}
\item The variational properties of smooth solutions of $(**)$,
\item The limiting behaviour of the family $\tilde{\Sigma}^{\varepsilon}_t$ of translating solutions of $(**)$ in $M\times\mathbb{R}$. 
\end{enumerate}
In particular, we show that the sets $E_t=\{u<t\}$, associated to a \textit{smooth} solution $u$ of $(**)$, minimise the following parametric energy functional 
\begin{equation}\label{functional1}
\mathcal{J}^{A}_{u,\nu}(F):=|\partial^* F\cap A|-\int_{F\cap A}|\nabla u|-\left(g^{ij}-\nu^i\nu^j\right)K_{ij},
\end{equation}
for each $t>0$. That is
\begin{equation}\label{varprincip2}
\mathcal{J}^A_{u,\nu}(E_t)\leq\mathcal{J}^A_{u,\nu}(F),
\end{equation}
for each set $F$ of locally finite perimeter that differs from the set $E_t$ on a compact subset $A$ of the domain. Here $P=\text{tr}_{\Sigma_t}K=(g^{ij}-\nu^i\nu^j)K_{ij}$, where $\nu$ represents the unit normal $\nabla u/|\nabla u|$ to the surfaces $\Sigma_t=\partial E_t$. The functional (\ref{functional1}), together with the minimisation principle (\ref{varprincip2}), generalises the variational formulation employed by Huisken and Ilmanen in \cite{HI01}, and accordingly allows the evolving surfaces to jump instantaneously over a positive volume at plateaus of the time-of-arrival function, $u$. However, in the weak setting, $\nabla u/|\nabla u|$ is undefined on plateaus of the locally Lipschitz function $u$, so in order to incorporate the extra $P$ term for this new flow, we must define an appropriate notion of normal vector in these jump regions. In this section we show that such a vector field can be obtained by taking an appropriate limit of the translating graphs $\tilde{\Sigma}^{\varepsilon}_t$. Since the null mean curvature of these surfaces is uniformly bounded, results of measure theory allow us to control them in $C^{1,\alpha}$, which leads to a foliation of the interior of the jump region $\{U=t_0\}=\{u=t_0\times\mathbb{R}\}$, at jump times $t_0$, by hypersurfaces satisfying the following result.
\begin{Proposition}\label{jumpprop}
Let $U(x,z)=u(x)$, where $u\in C^{0,1}_{\text{loc}}(M\backslash E_0)$ is the limit of the solution $u_{\varepsilon}$ of $(**)_{\varepsilon}$, as in (\ref{Udefn1}). Then the interior, $\tilde{\mathcal{K}}_{t_0}$, of the jump region $\{u=t_0\}\times\mathbb{R}=\{U=t_0\}$, at jump times $t_0$ is foliated by hypersurfaces with local uniform $C^{1,\alpha}$ estimates, where each such hypersurface is either a vertical cylinder or a graph over an open subset of $\{u=t_0\}$. Furthermore, each hypersurface bounds a Caccioppoli set that minimises $J_{U,\tilde{\nu}}$ in $\tilde{\mathcal{K}}_{t_0}$, where $\tilde{\nu}$ denotes the  $C^{0,\alpha}_{\text{loc}}$ normal vector field to the hypersurface foliation.
\end{Proposition}
The normal vector field $\tilde{\nu}$ to this foliation extends $\dfrac{\bar{\nabla}U}{|\bar{\nabla}U|}=\dfrac{(\nabla u,0)}{|\nabla u|}$  across the jump region 
$\tilde{\mathcal{K}}_{t_0}$ in $M\times\mathbb{R}$, and this extended vector field helps motivate the definition of weak solutions to $(**)$ in the next section. In this context, hypersurfaces and sets in $M\times\mathbb{R}$ will be denoted by the $\sim$ superscript for the remainder of this work, unless otherwise stated, and $\bar{\nabla}$ denotes the connection on $(M\times\mathbb{R},\bar{g})$.

To prove Proposition \ref{jumpprop} we utilise the following compactness result for sequences of minimisers of (\ref{functional1}). 
\begin{COMP2}\label{compactness}
Let $\tilde{\Omega}\subset M\times\mathbb{R}$ , and let $\tilde{E}_i\subset \tilde{\Omega}$ be a sequence of sets with $C^{1,\alpha}_{\text{loc}}$ boundary such that  $\partial \tilde{E}_i\to \partial \tilde{E},$ locally in $C^{1,\alpha}$, with outward unit normal $\nu_i\in C^{0,\alpha}_{\text{loc}}(T\tilde{\Omega})$ to $\partial \tilde{E}_i$ satisfying $\nu_i\to\nu$ locally uniformly.
Let $U_i\in C^{0,1}_{\text{loc}}(\tilde{\Omega})$ satisfy $U_i\to U$ locally uniformly, and assume that for each $\tilde{A}\subset\subset\tilde{\Omega}$, $\sup_{\tilde{A}}|\bar{\nabla} U_i|\leq C(\tilde{A})$ for large $i$. \\
If the sequence $\tilde{E}_i$ minimises $\mathcal{J}_{U_i,\nu_i}$ on $\tilde{\Omega}$, then $\tilde{E}$ minimises $\mathcal{J}_{U,\nu}$ in $\tilde{\Omega}$. 
\end{COMP2}
\begin{proof} We use the inequality 
\begin{equation}\label{compact}
\mathcal{J}_{U_i,\nu_i}(\tilde{E}_1)+\mathcal{J}_{U_i,\nu_i}(\tilde{E}_2)\geq \mathcal{J}_{U_i,\nu_i}(\tilde{E}_1\cup \tilde{E}_2)+\mathcal{J}_{U_i,\nu_i}(\tilde{E}_1\cap \tilde{E}_2),
\end{equation}
for an appropriate choice of Caccioppoli sets $\tilde{E}_1$ and $\tilde{E}_2$ such that $\tilde{E}_1\Delta \tilde{E}_2$ is precompact.\\

We first prove that $\tilde{E}$ minimises $\mathcal{J}_{U,\nu}$ on the outside in $\tilde{\Omega}$. To this end, consider $\tilde{F}\supset \tilde{E}$ with $\tilde{F}\backslash \tilde{E}\subset\subset\tilde{\Omega}$ and a suitable compact set $G\subset\tilde{\Omega}$ containing $\tilde{F}\backslash \tilde{E}$. Since the boundary of $G$ is not necessarily Lipschitz continuous, we consider a compact set $\bar{G}\subset\tilde{\Omega}$ with smooth boundary and $G\subset \text{int}(\bar{G})$ such that 
\begin{equation*}
 |\partial^*(\tilde{F}\cup \tilde{E}_i)\cap\partial\bar{G}|=|\partial^*(\tilde{F}\cap \tilde{E}_i)\cap\partial^*\bar{G}|=|\partial \tilde{E}_i\cap\partial\bar{G}|=0,
\end{equation*}
for all $i$, with traces satisfying $\int_{\partial\bar{G}}|\varphi^-_{\tilde{F}\cup \tilde{E}_i}-\varphi^+_{\tilde{E}_i}|d\mathcal{H}^{n+1}\to0$. This is possible because $\tilde{F}\cup \tilde{E}_i\to \tilde{E}$ and $\tilde{E}_i\to \tilde{E}$ in $L^1_{\text{loc}}(\tilde{\Omega}\backslash G)$. Then setting $\tilde{F}_i:=\tilde{E}_i\cup(\tilde{F}\cap \bar{G})$ we see that 
\begin{equation*}
 |\partial^*\tilde{F}_i\cap\tilde{\Omega}|=|\partial^*\tilde{E}_i\cap(\tilde{\Omega}\backslash\bar{G})|+|\partial^*(\tilde{F}\cup \tilde{E}_i)\cap\bar{G}|+\int_{\partial\bar{G}}|\varphi^-_{\tilde{F}\cup \tilde{E}_i}-\varphi^+_{\tilde{E}_i}|.
\end{equation*}
Now, since $\tilde{F}_i$ is an appropriate comparison function for $\tilde{E}_i$, we have \\$\mathcal{J}^{\bar{G}}_{U_i,\nu_i}(\tilde{E}_i)\leq \mathcal{J}^{\bar{G}}_{U_i,\nu_i}(\tilde{F}_i)$, implying
\begin{equation*}
 \mathcal{J}^{\bar{G}}_{U_i,\nu_i}(\tilde{E}_i)\leq \mathcal{J}^{\bar{G}}_{U_i,\nu_i}(\tilde{F}\cup \tilde{E}_i)+\int_{\partial\bar{G}}|\varphi^-_{\tilde{F}\cup \tilde{E}_i}-\varphi^+_{\tilde{E}_i}|.
\end{equation*}
Now inserting $\tilde{E}_1=\tilde{E}_i$ and $\tilde{E}_2=\tilde{F}$ into (\ref{compact}) we obtain
\begin{equation}\label{A1}
 \mathcal{J}^{\bar{G}}_{U_i,\nu_i}(\tilde{F})\geq \mathcal{J}^{\bar{G}}_{U_i,\nu_i}(\tilde{E}_i\cap \tilde{F})-\int_{\partial\bar{G}}|\varphi^-_{\tilde{F}\cup \tilde{E}_i}-\varphi^+_{\tilde{E}_i}|.
\end{equation}
Next we pass to limits. Since the trace term converges to zero, using lower semicontinuity we obtain
\begin{equation*}
\mathcal{J}^{\bar{G}}_{U,\nu}(\tilde{F})\geq \mathcal{J}^{\bar{G}}_{U,\nu}(\tilde{E}).
\end{equation*}

The fact that $\tilde{E}$ minimises $\mathcal{J}_{U,\nu}$ on the inside in $\tilde{\mathcal{G}}_{t_0}$ amongst competing sets $\tilde{F}\subset \tilde{E}$ satisfying $\tilde{E}\backslash \tilde{F}\subset\subset\tilde{\Omega}$ can similarly be proven by again constructing $\bar{G}$ and considering the comparison function $\tilde{F}_i:=\tilde{E}_i\cap \tilde{F}$ for $i>>1$ large enough.
\end{proof}

To prove Proposition \ref{jumpprop} we will also draw upon regularity theory for obstacle problems of the type (\ref{ofnal}) below. In particular, if the set $E_t:=\{u<t\}$ minimises $\mathcal{J}_{u,\nu}$, then it is almost minimal in the sense that 
\begin{equation}
|\partial^*E_t\cap B_R|\leq|\partial^*F\cap B_R|+C(n,\|Du\|_{\infty},\|K\|_{C^0})R^{n+1},
\end{equation}
for $E_t\Delta F\subset\subset B_R $. This means we can apply partial regularity results of geometric measure theory to obtain higher regularity for the level-sets $\Sigma_t=\partial E_t$. Specifically, we consider the following $C^{1,\alpha}$ result (see for example \cite{T84}), as quoted in \cite{HI01}. 
\begin{REG}\label{regthm}
Let $f$ be a bounded measurable function on a domain $\Omega$ with smooth metric $g$ and dimension $n+1<8$. Suppose $E$ contains an open set $A$ and minimises the functional 
\begin{equation}\label{ofnal}
 |\partial^*F|+\int_Ff
\end{equation}
 with respect to competitors $F$ such that $F\supseteq A,$ and $F\Delta E\subset\subset\Omega$.  
If $\partial A$ is $C^{1,\alpha}$, $0<\alpha\leq1/2,$ then $\partial E$ is a $C^{1,\alpha}$ submanifold of $\Omega$ with $C^{1,\alpha}$ estimates depending only on the distance to $\partial\Omega$, $\text{ess sup}|f|,\, C^{1,\alpha}$ bounds for $\partial A$, and $C^1$ bounds (including positive lower bounds) for the metric $g$. When $n+1\geq 8$, this remains true away from a closed singular set $Z$ of dimension at most $n-7$ that is disjoint from $\bar{A}$.
\end{REG}
\vskip 0.2 true in
\begin{jumpproof}
We break up the proof into the following Lemmata. 
\begin{Lemma}\label{regcorollary}
 The level-sets $\tilde{\Sigma}^{\varepsilon}_t=\{U_{\varepsilon}=t\}$ are locally uniformly bounded in $C^{1,\alpha}$.
\end{Lemma}
\begin{proof}
Since $\tilde{\Sigma}^{\varepsilon}_t=\{U_{\varepsilon}=t\}$ is a smooth solution of $(*)$ on $(M\backslash E_0)\times\mathbb{R}$, with smooth normal vector field $\nu_{\varepsilon}=\frac{\bar{\nabla}U_{\varepsilon}}{|\bar{\nabla}U_{\varepsilon}|}$, the functional $J_{U_{\varepsilon},\nu_{\varepsilon}}$ is well defined for sets $\tilde{F}\subset M\times\mathbb{R}$ of locally finite perimeter.  Using $\nu_{\varepsilon}$ as a calibration and applying the divergence theorem exactly as in the proof of Smooth Flow Lemma \ref{SMOO} shows that $\tilde{E}^{\varepsilon}_t:=\{U_{\varepsilon}<t\}$ minimises $\mathcal{J}_{U_{\varepsilon},\nu_{\varepsilon}}$ in $\Omega:=\tilde{E}^{\varepsilon}_b\backslash \tilde{E}^{\varepsilon}_a$ for $a\leq t\leq b$.

Now consider $\bar{x}=(x,x')\in(M\backslash\bar{E_0})\times\mathbb{R}$, and $d=\text{dist}(\bar{x},\partial E_0\times\mathbb{R})=\text{dist}(x,\partial E_0)$. Take $L'$ large enough that $B^M_{2d}(x)\subset F_{L'}$. Then for $\varepsilon\leq\varepsilon'=\varepsilon(L')$, (\ref{lipschitz}) provides a uniform bound for $|\nabla u_{\varepsilon}|$ (and thus also $|\bar{\nabla} U_{\varepsilon}|+P_{\tilde{\Sigma}^{\varepsilon}_t}$) on $B^{M\times\mathbb{R}}_{d}(\bar{x})$. It then follows from Theorem \ref{regthm} that the surfaces $\tilde{\Sigma}^{\varepsilon}_t\cap B^{n+1}_d(x)$ are uniformly bounded in $C^{1,\alpha}$ in $\varepsilon$ and $t$. 
 \end{proof}
\begin{Lemma}\label{foliation1}
Let $\tilde{\mathcal{K}}_{t_0}$ denote the interior of the jump region $\{U=t_0\}$, at a jump time $t_0$. Then each point $X_0=(x_0,z_0)\in \tilde{\mathcal{K}}_{t_0}$ lies in a surface $\tilde{\Sigma}_{X_0}\subset \tilde{\mathcal{K}}_{t_0}$ that is locally uniformly bounded in $C^{1,\alpha}$, and is either a vertical cylinder or a graph over an open subset of $\tilde{\mathcal{K}}_{t_0}$. \end{Lemma}
\begin{proof}
The sought-after surfaces are constructed using a pointwise approach similar to that used by Heidusch \cite{He01} to prove local uniform $C^{1,1}$ regularity estimates for the level-sets of the weak solution to inverse mean curvature flow. In particular, we fix a \textit{target point} 
\begin{equation}\label{target1}
X_0=(x_0,z_0)\in \tilde{\mathcal{K}}_{t_0},
\end{equation} 
and construct a surface containing that point. 

Given the convergent sequence $\varepsilon_i\to0$ that produces the limit $u$ of the elliptic regularised solution $u_{\varepsilon}$ as in (\ref{subseq}), we consider the corresponding sequence of times, $t_i,$ at which the surfaces $\tilde{\Sigma}^i_{t_i}=\text{graph}\left(\frac{u_i}{\varepsilon_i}-\frac{t_i}{\varepsilon_i}\right)$ pass through the target point $X_0$. This is possible because the translating graphs $\tilde{\Sigma}^i_t$ for $-\infty<t<\infty$ foliate $\Omega_i\times\mathbb{R}$, thus for every $i$ there is a unique $t_i$ such that $X_0\in \tilde{\Sigma}^i_{t_i}$. \\

In order to write each surface $\tilde{\Sigma}^{i}_{t_i}$ locally as a graph over its tangent space $T_{X_0}\tilde{\Sigma}^{i}_{t_i}$, we use the exponential map to work locally in normal coordinate charts on small Euclidean balls $B^{n+2}$. In particular, let $\iota(X_0)$ be the injectivity radius of $X_0$ in $M\backslash E_0\times\mathbb{R}$, and set 
\begin{equation}\label{ddd1}
d=d(X_0)=\min(\iota(X_0),\text{dist}(X_0,\partial \tilde{\mathcal{K}}_{t_0})).\end{equation}
By Corollary \ref{regcorollary} there exists $\varepsilon_0>0$ such that for all $t$ and $\varepsilon\leq\varepsilon_0$, the surface pieces $\tilde{\Sigma}^i_{t_i}\cap B^{M\times\mathbb{R}}_d(X_0)$ are uniformly $C^{1,\alpha}$ bounded in $t$ and $\varepsilon$. Now consider the exponential map
\begin{equation}\label{expo}
\exp_{X_0}=(\exp_{x_0},id_{\mathbb{R}}): T_{X_0}(M\times\mathbb{R})\cap B^{n+2}_d(0,z_0)\to B^{M\times\mathbb{R}}_d(X_0), 
\end{equation}
and set 
\begin{equation}\label{expsig}
\hat{\Sigma}^i_{t_i}=\exp^{-1}_{X_0}(\tilde{\Sigma}^i_{t_i}\cap B^{M\times\mathbb{R}}_d(X_0))\subset T_{X_0}(M\times\mathbb{R}).
\end{equation}
In the $\mathbb{R}$-direction the exponential map is just the identity, thus each surface $\hat{\Sigma}^{i}_{t_i}$ translates downwards in exactly the same manner as $\tilde{\Sigma}^i_{t_i}$. Furthermore, the surfaces $\hat{\Sigma}^i_{t_i}$ are uniformly $C^{1,\alpha}$ bounded in $t$ and $\varepsilon$.

Then there exists $R>0$, depending only on the locally uniform $C^{1,\alpha}$ bound, such that $ B^{n+2}_R(\hat{X}_0)\subseteq\hat{\Sigma}^{i}_{t_i}$ and thus the surface pieces $\hat{\Sigma}^{i}_{t_i}\cap B^{n+2}_R(\hat{X}_0)$ possess uniform $C^{1,\alpha}$ bounds. Here $\hat{X}_0=(\hat{x}_0,\hat{z}_0)=\exp_q^{-1}(X_0)$ is our target point.

The corresponding normals $\hat{\nu}_i(\hat{X}_0)$ to $\hat{\Sigma}^{i}_{t_i}\cap B^{n+2}_R(\hat{X}_0)$ create a sequence, a subsequence of which converges uniformly to a vector $\hat{\nu}(\hat{X}_0)$. The normal space to $\hat{\nu}(\hat{X}_0)$ defines a hyperplane $\hat{T}$ containing $\hat{X}_0$. Then by taking $i\gg1$ large enough, we can write the converging surfaces $\hat{\Sigma}^{i}_{t_i}\cap B^{n+2}_R(\hat{X}_0)$ as graphs of $C^{1,\alpha}$ functions $\hat{w}_i$ over $\hat{T}$. By reducing $R$, and taking $i\gg1$ large enough, we can then write each $\hat{\Sigma}^{i}_{t_i}$ locally as the graph of $\hat{w}_i$ over $\hat{T}\cap B^{n+1}_R(\hat{x}_0).$

 By Arzela-Ascoli, there exists a further subsequence $\hat{w}_{i_j}$ and a $C^1$ function $\hat{w}:\hat{T}\cap B^{n+1}_R(\hat{x}_0)\to\mathbb{R}$ such that 
 \begin{equation}\label{subsequence1}
 \hat{w}_{i_j}\rightarrow\hat{w}\quad\text{ in }C^{1}(\hat{T}\cap B^{n+1}_R(\hat{x}_0)).
 \end{equation} 
 Here $\hat{w}$ is locally the graph of a surface $\hat{\Sigma}_{\hat{X}_0}$ around $\hat{X}_0$, and $\hat{T}=T_{\hat{X}_0}\hat{\Sigma}_{X_0}$. Since the $C^{1,\alpha}$ bounds on $\hat{w}_i$ were independent of $i$, it follows that $\hat{w}\in C^{1,\alpha}(\hat{T}\cap B^{n+1}_R(\hat{x}_0))$, with the same uniform $C^{1,\alpha}$ bounds as $\hat{w}_i$. Thus $\exp_{q}(\hat{\Sigma}_{\hat{X}_0}):=\tilde{\Sigma}_{X_0}\cap B^{M\times\mathbb{R}}_d(X_0)$ is uniformly $C^{1,\alpha}$ bounded. By successively taking subsequences, the $\tilde{\Sigma}^i_{t_i}$ converge to a complete hypersurface that we will henceforth denote by $\tilde{\Sigma}_{X_0}$, since it coincides with $\tilde{\Sigma}_{X_0}$ near $X_0$.
 
Now $X_0\in\{U=t_0\}$ where, by hypothesis, $t_0:=\lim_{i\to\infty}t_i$ is a jump time. In order to argue that $\tilde{\Sigma}_{X_0}$ is contained in the set $\{U=t_0\}$, we note that it is a consequence of the above construction that any $y\in \tilde{\Sigma}_{X_0}$ is the limit of a sequence $y_i\in \tilde{\Sigma}^i_{t_i}$. The local uniform convergence $u_i\to u$ implies that $\hat{u}_i\to \hat{u}$ uniformly on $B^{n+1}_d(0)$. Thus the uniform convergence $U_i\rightarrow U$ on $B^{M\times\mathbb{R}}_d(X_0)$, together with the fact that $\lim_{i\to\infty}y_i=y$ then implies that $U(y)=t_0$, since
\begin{equation*}
 |U_i(y_i)-U(y)|\leq|U_i(y_i)-U_i(y)|+|U_i(y)-U(y)|\to 0,
\end{equation*}
and $\displaystyle\smash{\lim_{i\to\infty}}U^i(y_i)=\displaystyle\smash{\lim_{i\to\infty}}t_i=t_0,$
thus $\tilde{\Sigma}_{X_0}\subset\{U=t_0\}$. \\

This approach enables one to choose any point $X_0$ in the jump region $\tilde{\mathcal{K}}_{t_0}$ and construct the corresponding surface $\tilde{\Sigma}_{X_0}$ containing $X_0$. Since each $\tilde{\Sigma}_{X_0}$ is the limit of the graphs $\tilde{\Sigma}^i_{t_i}$ with local uniform $C^{1,\alpha}$ bounds, it is clear that each $\tilde{X}_0$ is either a vertical cylinder or a graph over an open subset of $\tilde{\mathcal{K}}_{t_0}\cap M$. Therefore, let $\Omega_G$ denote the open region in $\tilde{\mathcal{K}}_{t_0}\cap M$ where $|\nabla \hat{u}_i|$ converges locally uniformly to a finite limit, and let $\Omega_C$ denote the region where $|\nabla\hat{u}_i|$ converges to infinity. Then the translating nature of $\tilde{\Sigma}^{\varepsilon}_t$ together with the above construction dictates that the $\tilde{\Sigma}^i_{t_i}$ converge to a graph $\tilde{\Sigma}_{X_0}$ over $\Omega_G$, lying in a stack 
\begin{equation}\label{stack}
\{\tilde{\Sigma}_{X_{\alpha}}\}=\tilde{\Sigma}_{X_0}+\alpha \textbf{e}_{n+2},\quad\quad \alpha\in\mathbb{R},
\end{equation}
of vertical translates of $\tilde{\Sigma}_{X_0}$. To see this, note that 
\begin{equation*}X_0=(x_0,z_0)\in\tilde{\Sigma}^{i_j}_{t_{i_j}}=\text{graph}\left(\frac{u_{i_j}}{\varepsilon_{i_j}}-\frac{t_{i_j}}{\varepsilon_{i_j}}\right)\to\text{graph}(w)=\tilde{\Sigma}_{X_0},
\end{equation*}
\noindent implies $X_{\alpha}:=(x_0,z_0+\alpha)\in\tilde{\Sigma}^{i_j}_{t_{i_j}-\alpha\varepsilon_{ij}}$, where
\begin{equation*}\tilde{\Sigma}^{i_j}_{t_{i_j}-\alpha\varepsilon_{ij}}=\text{graph}\left(\frac{u_{i_j}}{\varepsilon_{i_j}}-\frac{t_{i_j}-\alpha\varepsilon_{i_j}}{\varepsilon_{i_j}}\right)\to\text{graph}(w)+\alpha\textbf{e}_z:=\tilde{\Sigma}_{X_\alpha},
\end{equation*}
where $w:=\exp_q(\hat{w})$. Therefore $\Omega_G\times\mathbb{R}$ is bounded by vertical cylinders, and filled by the stacks produced by the family $\{\tilde{\Sigma}_{X_{\alpha}}\}$ of vertical translations of each graph $\tilde{\Sigma}_{X_0}$.
\end{proof}

The possibility of two surfaces $\tilde{\Sigma}_{P_1}$ and $\tilde{\Sigma}_{P_2}$ from Lemma \ref{foliation1} touching tangentially at one point $P$, such that the outward unit normals agree at $P$ and $\tilde{\Sigma}_{P_1}$ lies outside $\tilde{\Sigma}_{P_2}$ (in the direction of the outward normal near $P$) is ruled out by the strong maximum principle. Furthermore, the intersection of two surfaces in the limit is ruled out by the translation invariance of the surfaces $\tilde{\Sigma}^{\varepsilon}_t$ and their local uniform $C^{1,\alpha}$ bounds.

We now argue that we can construct a ``normal" vector field $\tilde{\nu}$ on $\tilde{\mathcal{K}}_{t_0}$ using the surfaces from the proof of Lemma \ref{foliation1}. Since the limit surfaces are vertical cylinders or stacks of translation invariant graphs, the normal vector field $\tilde{\nu}$ to the family of surfaces in $\mathcal{K}_{t_0}$ is translation invariant, and we need only show that we can construct $\tilde{\Sigma}_{X_0}$ for each $X_0\subset\tilde{\mathcal{K}}_{t_0}\cap M$. Therefore, choose a dense set of points in $\tilde{\mathcal{K}}_{t_0}\cap M$. This corresponds to a countable set of points $\{p_i\}$, and for each such $p_i\in \tilde{\mathcal{K}}_{t_0}\cap M,$ we consider the convergent subsequence $\varepsilon_{i}$ such that $\tilde{\Sigma}^{i}_{t_{i}}$ converges to the hypersurface ${\tilde{\Sigma}_{P_i}}$ in $\tilde{\mathcal{K}}_{t_0}$, where $P_i:=(p_i,0)$. Then by taking a diagonal subsequence $\varepsilon_{i_*}$, we obtain local convergence of $\tilde{\Sigma}^{i_*}_{t_{i_*}}$ to $\
\tilde{\Sigma}_{P_i}$ for every point $p_i$ in the dense set.\\

Now consider a point $p_0\in \Omega_G$ such that $p_0$ is not in $\{p_i\}$. We wish to argue that we obtain local convergence to $\tilde{\Sigma}_{P_0}$ via the convergent sequence $\varepsilon_{i_*}$. There exists a point $p_i$ in the dense subset such that $\text{dist}(p_i,p_0)<d/10$. Let 
\begin{equation}
d_{P}:=\min(\iota(P),\text{dist}(P,\partial \tilde{\mathcal{K}}_{t_0})).
\end{equation}
By Corollary \ref{regcorollary}, the surfaces $\tilde{\Sigma}^{\varepsilon}_{t}$ are uniformly bounded in $B_{d_{p_i}}^{M\times\mathbb{R}}(P_i)$. Then since $B_{d_{P_i}/10}(P_0)\subset B_{d_{P_i}}(P_i)$, the surfaces $\tilde{\Sigma}^{\varepsilon}_{t}\cap B_{d_{P_i}/10}(P_0)$ possess the same uniform $C^{1,\alpha}$ bounds and we can take a convergent subsequence of $\varepsilon_{i_*}$ such that we obtain convergence to a limit surface $\tilde{\Sigma}_{P_0}$ in $B_{d_{P_i}/10}(P_0)$. Therefore this approach constructs a complete graph through each point $x_0\in\Omega_G$, and we obtain the vector field $\tilde{\nu}$ in all of $\Omega_G$.

Then given the uniform $C^{0,\alpha}$ normal vector field $\tilde{\nu}$ of the hypersurfaces constructed through the dense set of points $\{p_i\}$, we can extend the vector field $\tilde{\nu}$ to any points that have been missed in $\Omega_C$. Then translating $\tilde{\nu}$ in the $e_{n+2}$ direction, we obtain a normal vector field on the entire jump region $\tilde{\mathcal{K}}_{t_0}$. For the remainder of this work, let $\tilde{\nu}$ denote this translation invariant normal vector field to the surfaces $\tilde{\Sigma}_{X_0}$ foliating $\tilde{\mathcal{K}}_{t_0}$.

\begin{Lemma}\label{minimising}
Let $\tilde{\nu}$ denote the normal vector field to the surfaces foliating the jump region $\tilde{\mathcal{K}}_{t_0}$, as above. Then each surface $\tilde{\Sigma}_{X_0}$ in the jump region bounds a Caccioppoli set that minimises $\mathcal{J}_{U,\tilde{\nu}}$ in $\tilde{\mathcal{K}}_{t_0}$.
\end{Lemma}
\begin{proof}
Consider the Caccioppoli set $\tilde{E}$ that is bounded by the limit hypersurface $\tilde{\Sigma}_{X_0}$, such that $\tilde{\nu}$ is the outward unit normal of the relative boundary $\partial \tilde{E}\cap\tilde{\mathcal{K}}_{t_0}$. The sets $\tilde{E}^{i}_{t_{i}}$ minimize the functional $J_{U_i,\nu_i}$ in $\tilde{\mathcal{K}}_{t_0}$, where $\nu_i=\frac{\bar{\nabla}U_i}{|\bar{\nabla}U_i|}$. Passing these sets to limits as in the proof of Lemma \ref{foliation1} to obtain the limit surface $\tilde{\Sigma}_{X_0}$, Theorem \ref{compactness} then says that $\tilde{E}$ minimises $\mathcal{J}_{U,\tilde{\nu}}$ in $\tilde{\mathcal{K}}_{t_0}$. 
\end{proof}

Collecting the above results, we obtain a family of $C^{1,\alpha}_{\text{loc}}$ hypersurfaces foliating $\Omega_G\times\mathbb{R}$, and by extending the family of cylindrical hypersurfaces in $\Omega_C\times\mathbb{R}$ to any missed points in $\Omega_C$, we obtain a foliation of the entire interior region $\tilde{\mathcal{K}}_{t_0}$. At each point $X_0=(x_0,t_0)$, the corresponding leaf of the foliation passing through $X_0$ is constructed by taking the limit of the $\Sigma^{\varepsilon}_t$ locally around $X_0$, as in Lemma \ref{foliation1}. This completes the proof.\end{jumpproof}

\section{Variational formulation of weak solutions.}\label{2.3}
By freezing $|\nabla u|-\text{tr}_{\Sigma_t}K$ and treating it as a bulk term, one may interpret $(**)$ as the Euler-Lagrange equation of the functional
\begin{equation}\label{functional}
\mathcal{J}_{u,\nu}(v):=\int|\nabla v|+v\left(|\nabla u|-\left(g^{ij}-\nu^i\nu^j\right)K_{ij}\right)dx.
\end{equation}
For a smooth family of solutions of $(*)$, we will see below that the corresponding time-of-arrival function $u$ defined by (\ref{time of arrival}) satisfies 
\begin{equation}\label{Mvar}
\mathcal{J}_{u,\nu}(u)\leq\mathcal{J}_{u,\nu}(v),
\end{equation}
among competing locally Lipschitz functions $v$, that differ from $u$ on a compact subset of $M\backslash\bar{E_0}$. The relationship between the variational formulation (\ref{Mvar}) and the functional (\ref{functional1}) is then given by the following Lemma.
\begin{Lemma}\label{equiv}
Let $u$ be a locally Lipschitz function in the open set $\Omega$, and $\nu$ a measurable vector field on $T\Omega$. Then $u$ satisfies (\ref{Mvar}) on $\Omega$ if and only if for each $t$, $E_t:=\{u<t\}$ minimizes (\ref{functional1}) in $\Omega$.
\end{Lemma}
\begin{proof}
This follows exactly as in \cite[Lemma 1.1]{HI01}, with $|\nabla u|-(g^{ij}-\nu^i\nu^j)K_{ij}$ replacing the bulk term $|\nabla u|$.\\ 
\end{proof} 
This equivalence between the two variational formulations also extends to the initial value problem
\begin{equation}\label{initial}
 \begin{aligned}u\in &C^{0,1}_{\text{loc}}(M),\,\,\nu \text{ a measurable vector field on }T(M\backslash E_0),\\
 &E_0=\{u<0\},\text{ and } u\text{ satisfies }(\ref{Mvar}) \text{ in } M\backslash E_0.\\
\end{aligned}
\end{equation}
To see this equivalence, let $E_t$ be a nested family of open sets in $M$, closed under ascending union, and define $u$ as in the statement of Lemma \ref{equiv} by the characterisation $E_t=\{u<t\}$. Then using Lemma \ref{equiv} and approximating up to the boundary, we see that (\ref{initial}) is equivalent to 
\begin{equation}\label{initial2}
\begin{aligned}
u&\in C^{0,1}_{\text{loc}}M,\,\,\nu \text{ a measurable vector field on }T(M\backslash E_0)\\
&\text{ and }E_t\text{ minimises }J_{u,\nu}\text{ in }M\backslash E_0\text{ for each }t>0.
\end{aligned}
\end{equation}
Lastly, by approximating $s\searrow t$, we see that (\ref{initial}) and (\ref{initial2}) are equivalent to
\begin{equation}\label{1.9}
\begin{aligned}
&u\in C^{0.1}_{\text{loc}}(M),\,\,\nu \text{ a measurable vector field on }T(M\backslash E_0)\\
&\text{ and }\{u\leq t\}\text{ minimises }J_{u,\nu}\text{ in } M\backslash E_0\text{ for each }t\geq0.
\end{aligned}
\end{equation}
\vskip 0.1 true in

We now present the precise definition of weak solutions to $(**)$. In the previous section we highlighted the need to define the normal vector field $\nu$ in jump regions in order to incorporate the $P=(g^{ij}-\nu^i\nu^j)K_{ij}$ term into a variational formulation of weak solutions to $(**)$. We showed that taking an appropriate limit of the smooth translating solutions $\tilde{\Sigma}^{\varepsilon}_t=\{U_{\varepsilon}=t\}$ of $(*)$, provides a constructive method of foliating the interior $\tilde{\mathcal{K}}_{t_0}$ of the jump region $\{U=t_0\}=\{u=t_0\}\times\mathbb{R}$, one dimension higher, by $C^{1,\alpha}_{\text{loc}}$ hypersurfaces $\tilde{\Sigma}_{X_0}$ in $M\times\mathbb{R}$ with uniform $C^{0,\alpha}_{\text{loc}}$ unit normal vector field $\tilde{\nu}$. Each such hypersurface $\tilde{\Sigma}_{X_0}$ in the foliation is either (part of) a vertical cylinder, or is a smooth graph over an open subset of $\tilde{\mathcal{K}}_{t_0}$, in the stack 
\begin{equation}\label{stack2}
\tilde{\Sigma}+\alpha\,\textbf{e}_{n+2},\quad \alpha\in\mathbb{R},
\end{equation}
\noindent of vertical translates of $\tilde{\Sigma}$. The normal vector field $\tilde{\nu}$ to each vertical cylinder is perpendicular to the $z$-direction, and could therefore be projected to $M$ without loss of information. However, in the case of the graphical hypersurfaces (\ref{stack2}), information would be lost if one were to define the vector field $\nu$ in (\ref{functional1}) to be the projection of $\tilde{\nu}$ to $TM$. 

This motivates the choice to formulate the weak solution to $(**)$ one dimension higher, in terms of a translation invariant function $U(x,z)=u(x)\in C^{0,1}_{\text{loc}}(M\times\mathbb{R})$, and a translation invariant vector field $\bar{\nu}\in C^{0,\alpha}_{\text{loc}}(T((M\backslash E_0)\times\mathbb{R}))$ that extends $\bar{\nabla}U/|\bar{\nabla}U|$ across the jump region. One then considers the analogously defined functionals $\mathcal{J}_{U,\bar{\nu}}$ to (\ref{functional}) and (\ref{functional1}) for such pairs $(U,\bar{\nu})$ in $M\times\mathbb{R}$, and we remark that Lemma \ref{equiv} and the initial value problem equivalences (\ref{initial})-(\ref{1.9}) hold in $M\times\mathbb{R}$ (for general $U$ and $\bar{\nu}$ that are not necessarily translation invariant, like we will demand for the weak solution of $(**)$). 

In Lemma \ref{minimising} we showed that each of the surfaces $\tilde{\Sigma}_{X_0}$ foliating the jump region $\tilde{\mathcal{K}}_{t_0}$ bounds a Caccioppoli set that minimises $J_{U,\tilde{\nu}}$ in the jump region $\tilde{\mathcal{K}}_{t_0}$. Together with Lemma \ref{equiv}, this motivates the restriction in Definition \ref{weakdefinition} below that at each point $X\in (M\backslash\bar{E}_0)\times\mathbb{R}$, $\tilde{\nu}(X)$ be the normal vector to a $C^{1,\alpha}$ hypersurface that bounds a Caccioppoli set minimising $J_{U,\bar{\nu}}$ in $(M\backslash E_0)\times\mathbb{R}$.
\begin{Definition}\label{weakdefinition}
Let  $E_0\subset M$ be a precompact, open set with $C^2$ boundary $\Sigma_0=\partial E_0$. We call the pair $(U,\bar{\nu})$ a weak solution of $(**)$ with initial condition $E_0$ if $U\in C^{0,1}_{\text{loc}}(M\times\mathbb{R})$ and $\bar{\nu}\in C^{0,\alpha}_{\text{loc}}(T(M\backslash E_0)\times\mathbb{R})$ satisfy
\begin{enumerate}
\item[$(i)$] $U$ is translation invariant in the vertical direction. In particular, there exists a locally Lipschitz function $u:M\to\mathbb{R}$ such that $U(x,z)=u(x)$ and $u$ satisfies
\begin{equation*}
\begin{aligned}
&\cdot u(x)\geq0\,\,\forall\, x\in M\backslash E_0, \quad\quad\quad\quad\quad\quad\\
&\cdot u\big|_{\partial E_0}=0,\quad u(x)<0\quad\forall x\in E_0,\\
&\cdot u(x)\to+\infty\,\,\text{as dist}(x,E_0)\to\infty.
\end{aligned}
\end{equation*}
\item[$(ii)$] The set $\tilde{E}_t=\{U<t\}$ minimises $J_{U,\bar{\nu}}$ in $(M\backslash E_0)\times\mathbb{R}$ for each $t>0$. At jump times $t_0$, each point $X_0=(x_0,z_0)$ in the interior $\tilde{\mathcal{K}}_{t_0}$ of the jump region $\{U=t_0\}$ lies in the boundary $\partial \tilde{E}_{X_0}\in C^{1,\alpha}_{\text{loc}}$ of a Caccioppoli set $\tilde{E}_{X_0}$ that minimises $J_{U,\bar{\nu}}$\ in $\tilde{\mathcal{K}}_{t_0}$.
\item[$(iii)$] $\bar{\nu}$ is a translation invariant, unit vector field such that  
\begin{equation*}
\begin{aligned}
&\cdot\bar{\nu}(X+\alpha\,\textbf{e}_{z})=\bar{\nu}(X)\quad\forall\, X\in (M\backslash E_0)\times\mathbb{R},\,\,\alpha\in\mathbb{R},\\
&\cdot \bar{\nu}(X)\text{ is the normal vector to }\partial \tilde{E}_t\text{ at each point }X\in\partial \tilde{E}_t, \\
&\cdot  \bar{\nu}(X)\text{ is the normal vector to }\partial \tilde{E}_{X_0}\text{ at each point }X\in\partial \tilde{E}_{X_0},\\
&\,\,\text{ at jump times }t_0.
\end{aligned}
\end{equation*}
\end{enumerate}
\end{Definition}
\textbf{Remarks} 1. Unlike in the weak formulation of inverse mean curvature flow, which asks only that $E_t=\{u<t\}$ minimise $J_{u,\nu}$ for each $t>0$, we require the variational principle (\ref{varprincip2}) for $J_{U,\bar{\nu}}$ to be satisfied \textit{everywhere}, in particular in the interior of the jump region.\\
2. By Lemma \ref{equiv}, any weak solution $(U(x,z):=u(x),\nu)$ of $(**)$ satisfies (\ref{initial}) on $(M\backslash\bar{E_0})\times\mathbb{R}.$ Furthermore, we find that $(u,\nu_M:=\bar{\nu}\big|_{TM})$ satisfies (\ref{Mvar}) in $M\backslash\bar{E}_0$.

\begin{Lemma}\label{reduction}
Let $(U(x,z):=u(x),\bar{\nu})$ be a weak solution of $(**)$ with initial condition $E_0$. Then the pair $(u,,\nu_M)$ satisfies (\ref{Mvar}) on $M\backslash \bar{E}_0$, and $E_t=\{u<t\}$ minimises $J_{u,\nu_M}$ for each $t>0$, where $\nu_M:=\bar{\nu}\big|_{TM}.$
\end{Lemma}
\begin{proof}
Since the tensor $K$ is extended trivially in the $z$-direction, we find
\begin{equation}
\left(\bar{g}^{ij}-\bar{\nu}^i\bar{\nu}^j\right)K_{ij}=\left(g^{ij}-\nu_M^i\nu_M^j\right)K_{ij},
\end{equation}
where $\nu_M:=\bar{\nu}\big|_{TM}$. Let $B_{u,\nu_M}:=|\nabla u|-\left(g^{ij}-\nu_M^i\nu_M^j\right)K_{ij}$ denote the bulk term of $\mathcal{J}_{u,\nu_M}$.
Let $v$ be a locally Lipschitz function such that $\{v\neq u\}\subset A\subset\subset M\backslash \bar{E}_0$. Let $\phi(z)$ be a cutoff function such that:
\begin{equation*}
|\phi_z|\leq2,\quad\phi=1\text{ on }[0,s]\text{ and }\phi=0\text{ on }\mathbb{R}\backslash(-1,s+1).
\end{equation*}
Then $V(x,z):=\phi(z)v(x)+(1-\phi(z))u(x)$ is an appropriate comparison function for $U$, and letting $\tilde{A}:=A\times[-1,s+1]$, we obtain from (\ref{Mvar})
\belowdisplayskip=0pt
\begin{align*}
\int_{\tilde{A}}|\nabla u|+u\,B_{u,\nu_M}&=\int_{\tilde{A}}|\bar{\nabla}U|+U\left(|\bar{\nabla}U|-\left(\bar{g}^{ij}-\bar{\nu}^i\bar{\nu}^j\right)K_{ij}\right)\\
&\leq\int_{\tilde{A}}|\bar{\nabla}V|+V\left(|\bar{\nabla}U|-\left(\bar{g}^{ij}-\bar{\nu}^i\bar{\nu}^j\right)K_{ij}\right)\\
&\leq\int_{\tilde{A}}\phi\left(|\nabla v|+v\,B_{u,\nu_M}\right)+(1-\phi)\left(|\nabla u|+u\,B_{u,\nu_M}\right)\\
&\,\,\,\,\,\,\,+|\phi_z|||v-u|.
\end{align*}
This implies
\begin{align*}
s\cdot\mathcal{J}^A_{u,\nu_M}(u)&=s\int_A|\nabla u|+u\,B_{u,\nu_M}\quad\quad\quad\quad\textcolor{white}{.}\\
&\leq \int_{\tilde{A}}\phi\left(|\nabla u|+u\,B_{u,\nu_M}\right)\\
&\leq\int_{\tilde{A}}\phi\left(|\nabla v|+v\,B_{u,\nu_M}\right)+|\phi_z|||v-u|\\
&\leq(s+2)\int_A|\nabla v|+v\,B_{u,\nu_M}+\int_{A\times([-1,0]\subset[1,2])}|\phi_z||v-u|\\
&\leq(s+2)\mathcal{J}^A_{u,\nu_M}(v)+4\int_A|v-u|.\\
\end{align*}
Dividing by $s$ and passing $s\to\infty $ proves that the pair $(u,\nu_M)$ satisfies (\ref{Mvar}). Lemma \ref{equiv} then implies that the sets $E_t:=\{u<t\}$ minimise $J_{u,\nu_M}$ for each $t>0$.
\end{proof}
We now state some further properties of weak solutions of $(**)$. We begin by showing that smooth solutions of the flow $(*)$ are weak solutions in the domain they foliate. This follows as in \cite[Lemma 2.3]{HI01}.
\begin{SMOO}\label{SMOO} Let $(\Sigma_t)_{a\leq t< b}$ be a smooth solution of $(*)$ on $M$. Let $U=t$ on $\Sigma_t\times\mathbb{R},\,U<a$ in the region bounded by $\Sigma_a\times\mathbb{R}$, and $\tilde{E}_t:=\{U<t\}$. Then $\tilde{E}_t$ minimises $\mathcal{J}_{U,\bar{\nu}}$ in $\tilde{E}_b\backslash \tilde{E}_a$ for $a\leq t<b$, where $\bar{\nu}$ is the smooth normal to the vertical cylinder $\Sigma_t\times\mathbb{R}$, given by $\bar{\nu}=(\nu_{\Sigma_t},0)=\dfrac{\bar{\nabla}U}{|\bar{\nabla}U|}$.
\end{SMOO}
\begin{proof}
We use the smooth normal $\tilde{\nu}=\dfrac{\bar{\nabla} U}{|\bar{\nabla} U|}$ as a calibration and apply the divergence theorem to relate $\mathcal{J}_{U,\tilde{\nu}}(\tilde{E}_t)$ to $\mathcal{J}_{U,\bar{\nu}}(\tilde{F})$ for a competing set $\tilde{F}$ of finite perimeter with $\tilde{F}\Delta \tilde{E}_t\subset\subset\tilde{\Omega}$. Let $B_{U,\bar{\nu}}:=|\bar{\nabla} U|-(\bar{g}^{ij}-\bar{\nu}^i\bar{\nu}^j)K_{ij}$ denote the bulk energy term in $\mathcal{J}_{U,\bar{\nu}}$.
\begin{align*}
\mathcal{J}_{U,\bar{\nu}}(\tilde{E}_t)=&|\partial \tilde{E}_t|-\int_{\tilde{E}_t}B_{U,\bar{\nu}}\,dx=\int_{\partial \tilde{E}_t}\nu_{\partial {\tilde{E}_t}}\cdot\bar{\nu} d\mathcal{H}^{n-1}-\int_{E_t}B_{U,\bar{\nu}}\,dx\\
         =&\int_{\partial \tilde{E}_t\cap\bar{\tilde{F}}}\nu_{\partial_{\tilde{E}_t}}\cdot\bar{\nu} d\mathcal{H}^{n-1}+\int_{\partial \tilde{E}_t\backslash \tilde{F}}\nu_{\partial {\tilde{E}_t}}\cdot\bar{\nu} d\mathcal{H}^{n-1}-\int_{\tilde{E}_t}B_{U,\bar{\nu}}\,dx\\
=&\int_{\partial^*F\cap\bar{\tilde{E}}_t}\nu_{\partial^*\tilde{F}}\cdot\bar{\nu} d\mathcal{H}^{n-1}+\int_{\tilde{E}_t\backslash \tilde{F}}B_{U,\bar{\nu}}\,dx-\int_{\tilde{E}_t}B_{U,\bar{\nu}}\,dx\\
&+\int_{\partial^*\tilde{F}\backslash \tilde{E}_t}\nu_{\partial^*\tilde{F}}\cdot\bar{\nu} d\mathcal{H}^{n-1}-\int_{\tilde{F}\backslash \tilde{E}_t}B_{U,\bar{\nu}}\,dx\\
=&\int_{\partial^*\tilde{F}}\nu_{\partial^*\tilde{F}}\cdot\bar{\nu} d\mathcal{H}^{n-1}-\int_{\tilde{F}}B_{U,\bar{\nu}}dx\leq|\partial^*\tilde{F}|-\int_{\tilde{F}}B_{U,\bar{\nu}}\,dx=\mathcal{J}_{U,\bar{\nu}}(\tilde{F}).\end{align*}
\end{proof}
\textbf{Weak Mean Curvature.}
In view of the local $C^{1,\alpha}$ estimates given by Regularity Theorem \ref{regthm}, we can consider the weak mean curvature of the surfaces $\tilde{\Sigma}_t=\partial\{U<t\}$. 

Let $\tilde{N}$ be a $C^1$ hypersurface in $M\times\mathbb{R}$. Then a locally integrable function $H$ on $\tilde{N}$ is called the weak mean curvature provided 
\begin{equation}\label{wkHdefn}
\int_{\tilde{N}}\text{div}_{\tilde{N}}Xd\mu=\int_{\tilde{N}}H\nu\cdot Xd\mu,\quad\quad\forall X\in C^{\infty}_c(T(M\times\mathbb{R})).
\end{equation}

\begin{Lemma}
Let $\tilde{E}_t:=\{U<t\}$ minimise $J_{U,\bar{\nu}}$ in $\tilde{A}:=\tilde{E}_b\backslash \tilde{E}_a$, for $U\in C^{0,1}_{\text{loc}}(\tilde{A})$. Then the surfaces $\tilde{\Sigma}_t=\partial \tilde{E}_t$ have weak mean curvature $H$ satisfying $H=|\bar{\nabla} U|-P$ for a.e. $x\in \tilde{\Sigma}_t$ and a.e. $t\in(a,b)$, where $P=(\bar{g}^{ij}-\bar{\nu}^i\bar{\nu}^j)K_{ij}$.
\end{Lemma}
\begin{proof}
Let $X$ be a compactly supported vector field defined on $M$, and $(\Phi_s)_{-\varepsilon<s<\varepsilon}$ the flow of diffeomorphisms generated by $X$ with $\Phi_0= id_M$. For minimisers of $\mathcal{J}_{U,\nu}$, we use the area formula, the dominated convergence theorem and the co-area formula in the form 
\begin{equation*}
\int_{\mathbb{R}^{n+2}} |\bar{\nabla}f|dx=\int_{-\infty}^{\infty}\int_{\{f=t\}}dt
\end{equation*}
to obtain
\begin{align*}
0&=\dfrac{d}{ds}\bigg|_{s=0}\mathcal{J}_{U,\bar{\nu}}(U\circ\Phi_s^{-1})\\
  &=\dfrac{d}{ds}\bigg|_{s=0}\left(\int_{\tilde{W}}|\nabla (U\circ\Phi_s^{-1})|+(U\circ\Phi_s^{-1})\left(|\bar{\nabla} U|-(\bar{g}^{ij}-\bar{\nu}^i\bar{\nu}^j)K_{ij}\right)dx\right)\\
  &=\dfrac{d}{ds}\bigg|_{s=0}\left(\int_{-\infty}^{\infty}\int_{\tilde{\Sigma}_t\cap \tilde{W}}|\det d\Phi_s(x)|d\mathcal{H}^{n}(x)dt\right)\\
  &\quad-\int_{\tilde{W}}\bar{\nabla} U\cdot X\left(|\bar{\nabla} U|-(\bar{g}^{ij}-\bar{\nu}^i\bar{\nu}^j)K_{ij}\right)dx\\
  &=\int_{-\infty}^{\infty}\int_{\tilde{\Sigma}_t\cap \tilde{W}}\text{div}_{\tilde{\Sigma}_t}Xd\mathcal{H}^ndt-\int_{\tilde{W}}\bar{\nu}\cdot X|\bar{\nabla} U|\left(|\bar{\nabla} U|-(\bar{g}^{ij}-\bar{\nu}^i\bar{\nu}^j)K_{ij}\right)dx,
\end{align*}
since $\bar{\nu}=\dfrac{\bar{\nabla} U}{|\bar{\nabla} U|}$ when $\bar{\nabla} U\neq0$, and $\bar{\nu}|\bar{\nabla} U|=0$ when $\bar{\nabla} U=0$. Then by the co-area formula, we obtain
\begin{align*}
0=\int_{-\infty}^{\infty}\int_{\tilde{\Sigma}_t\cap \tilde{W}}\left(\text{div}_{\tilde{\Sigma}_t}X+(P-|\bar{\nabla} U|)\bar{\nu}\right)\cdot X d\mathcal{H}^{n+1}dt.
\end{align*}
Lebesgue differentiation and comparison with (\ref{wkHdefn}) yields the result.
\end{proof}
Exactly as in the proof of \cite[Theorem 2.1]{HI01}, we also obtain the following compactness theorem for the time-of-arrival function.
\begin{COMP}\label{COMP} Let $U_i\in C^{0,1}_{\text{loc}}(\tilde{\Omega}_i)$ and $\bar{\nu}_i\in C^{0,\alpha}_{\text{loc}}(T\tilde{\Omega}_i)$ be a sequence of solutions of (\ref{Mvar}) on open sets $\tilde{\Omega}_i\subset M\times\mathbb{R}$, such that
\begin{equation}\label{conditions}
\tilde{\Omega}_i\to\tilde{\Omega},\quad U_i\to U,\quad\bar{\nu}_i\to\bar{\nu},
\end{equation}
locally uniformly, and such that for each $\tilde{A}\subset\subset\tilde{\Omega}$, $\sup_{\tilde{A}}|\bar{\nabla} U_i|\leq C(\tilde{A}), $
for large $i$, where $C(\tilde{A})$ is independent of $i$. Then $(U,\bar{\nu})$ solves (\ref{Mvar}) on $\tilde{\Omega}$. In the special case where $(U_i,\bar{\nu}_i)$ is a sequence of weak solutions of $(**)$ satisfying Definition \ref{weakdefinition}, then the limit $(U,\bar{\nu})$ is a weak solution of $(**)$.
\end{COMP}
\section{Geometric characterisation of jump regions}\label{jjump}
In this section we introduce the concept of outward optimisation in order to give a geometric characterisation of the criterion selecting jump times. 
Since weak solutions $(U(x,z)=u(x),\bar{\nu})$ of $(**)$ are translation invariant and the level sets of $U$ are vertical cylinders, this characterisation follows from the parametric variational formulation (\ref{varprincip2}) for $(u,\nu_M:=\bar{\nu}\vert_{TM})$. 

Let $\Omega$ be an open set in $M$. Then we call a set $E$ \textit{outward optimising $($in $\Omega)$ with respect to $\nu$}, if $E$ minimises `area plus bulk energy $P$' on the outside in $\Omega$. That is, if 
\begin{equation}\label{mini}
|\partial^*E\cap A|\leq |\partial^*F\cap A| +\int_{F\backslash E}  (g^{ij}-\nu^i\nu^j)K_{ij},
\end{equation}
for any $F$ containing $E$ such that $F\backslash E\subset\subset \Omega$, and any compact set $A$ containing $F\backslash E$. Here $\nu$ is a measurable vector field on $F\backslash E$. The set $E$ is then called \textit{strictly outward optimising $($in $\Omega)$} if equality in (\ref{mini}) implies that $F\cap\Omega=E\cap\Omega$ a.e.\\

We use this concept to define the strictly outward optimising hull of a measurable set $E\subset\Omega$. Specifically, we define $E'=E'_{\Omega}$ to be the intersection of the Lebesgue points of all the strictly outward optimising sets in $\Omega$ that contain $E$. We call $E'$ the \textit{strictly outward optimising hull of $E$ $($in $\Omega)$}. Up to a set of measure zero, $E'$ may be realised by a countable intersection, so $E'$ is strictly outward optimising, and open.\\ 

We then obtain the following interpretation of the variational formulation. 
\begin{OUT}\label{outward1}
Suppose that $(U(x,z):=u(x),\bar{\nu})$ is a weak solution of $(**)$ with initial condition $E_0$, and that $M$ has no compact components. Then:
\begin{enumerate}
 \item[$(i)$] For $t>0$, $E_t$ is outward optimising in $M$ with respect to $\nu_M:=\bar{\nu}\big|_{TM}$.
  \item[$(ii)$]  For $t\geq0$, $E_t^+$ is strictly outward optimising in $M$ with respect to $\nu_M$.
   \item[$(iii)$] For $t\geq0$, $E_t'=E_t^+$, provided $E_t^+$ is precompact.
   \item[${(iv)}$] For $t>0$, $|\partial E_t|=|\partial E_t^+|+\int_{E_t^+\backslash E_t}(g^{ij}-\nu_M^i\nu_M^j)K_{ij}$, provided that $E_t^+$ is precompact. This extends to $t=0$ precisely if $E_0$ is outward optimising.
  \end{enumerate}
Furthermore, for general $(U,\bar{\nu})$ satisfying (\ref{initial}) in $M\times\mathbb{R}$, the analogous statements hold on compact sets $\tilde{\Omega}\subset M\times\mathbb{R}$ with 
\begin{equation}\label{mini2}
|\partial^*\tilde{E}\cap \tilde{A}|\leq |\partial^*\tilde{F}\cap \tilde{A}| +\int_{\tilde{F}\backslash \tilde{E}}  (\bar{g}^{ij}-\bar{\nu}^i\bar{\nu}^j)K_{ij},
\end{equation}
replacing (\ref{mini}) in the definition of outward optimising.  
\end{OUT}
To prove Outward Optimising Property \ref{outward1}, we will need the following Lemma.
\begin{Lemma}\label{Conn} Let $(U,\bar{\nu})$ satisfy (\ref{Mvar}) on $\tilde{\Omega}$. Then $U$ has no strict local maxima or minima on $\tilde{\Omega}$.\end{Lemma}
\begin{proof}
First assume that $U$ possesses a strict local maximum so that there is a connected, precompact component $\tilde{E}$ of $\{U>t\}$ for some $t$. Define the Lipschitz function $V_k$ by
\begin{equation}\label{v_k}
V_k:=\begin{cases}\begin{aligned}&k\quad\text{on }\hat{E}_k:=\tilde{E}_k\cap \tilde{E},\\
                                                              &U\quad\text{on }\tilde{\Omega}\backslash \hat{E}_k,
                                                                                         \end{aligned}\end{cases}
\end{equation}                                                                                      
for $0<k<\sup\limits_{\tilde{E}} U$ and $\tilde{E}_k:=\{U>k\}.$ Then (\ref{Mvar}) and H\"older's inequality yield
\small\begin{align}
\label{11}\int_{\hat{E}_k}|\bar{\nabla} U|(1+U-k)&\leq \int_{\hat{E}_k}(U-k)C_0\leq C_0\left(\int_{\hat{E}_k}(U-k)^{\frac{n}{n-1}}\right)^{\frac{n-1}{n}}|\hat{E}_k|^{\frac{1}{n}},
\end{align}\normalsize
where $C_0=(n+1)|\lambda|$ and $|\lambda|$ is the size of the largest eigenvalue of $K$ on $\tilde{\Omega}$. Then using the Sobolev inequality on the left hand side we obtain
\small\begin{equation}\label{22}
\int_{\hat{E}_k}|\bar{\nabla} U|(1+U-k)\geq\int_{\hat{E}_k}|\bar{\nabla} U|= \int_{\hat{E}_k}|\bar{\nabla} (U-k)|\geq\left(\int_{\hat{E}_k}(U-k)^{\frac{n}{n-1}}\right)^{\frac{n-1}{n}}.
\end{equation}\normalsize                                                     
Combining (\ref{11}) and (\ref{22}) we find $1\leq C_0|\hat{E}_k|^{1/n}$, which leads to a contradiction since $|\hat{E}_k|$ can be made arbitrarily small by choosing $k$ close to $\sup\limits_{\tilde{E}}U$.

Now assume that $U$ possesses a strict local minimum and let $\tilde{E}$ be a connected, precompact component of $\{U<t\}$ for some $t$, and again consider the function $V_k$ defined by (\ref{v_k}), where this time $k>\inf\limits_{\tilde{E}} U$ and $\tilde{E}_k:=\{U<k\}.$ Then as above, (\ref{Mvar}) and H\"older's inequality yield
\begin{align}\label{33}
\int_{\tilde{E}_k}|\bar{\nabla} U|(1+U-k)&\leq C_0\left(\int_{\tilde{E}_k}(U-k)^{\frac{n}{n-1}}\right)^{\frac{n-1}{n}}|\tilde{E}_k|^{\frac{1}{n}},
\end{align}
and by restricting to $k$ small enough that $1+U-k\geq\frac{1}{2}$ on $\tilde{E_k}$, we obtain
\begin{equation}\label{44}
\int_{\tilde{E}_k}|\bar{\nabla} U|(1+U-k)\geq\frac{1}{2}\left(\int_{\tilde{E}_k}(U-k)^{\frac{n}{n-1}}\right)^{\frac{n-1}{n}}.
\end{equation}                                                     
Combining (\ref{33}) and (\ref{44}) we find $1/2\leq C_0|\tilde{E}_k|^{1/n}$, which leads to a contradiction since $|\tilde{E}_k|$ can be made arbitrarily small by choosing $k$ close to $\inf\limits_EU$.\\
In the case where $(U(x,z):=u(x),\bar{\nu})$ is a weak solution of $(**)$, repeating the above calculation for $u$ on $M$, using (\ref{reduction}), yields the desired result.\end{proof}

\begin{outwardproof}
Refer to \cite[Minimising Hull Property 1.4]{HI01}.\\
(\textit{i}) This follows immediately from Lemma \ref{reduction}.\\
(\textit{ii}) From (\ref{1.9}) we obtain for suitable $A$
\begin{equation}\label{1.17}
|\partial^*E_t^+\cap A|\leq|\partial^*F\cap A|+\int\limits_{F\backslash E_t^+}(g^{ij}-\nu_M^i\nu_M^j)K_{ij}-|\nabla u|dx,
\end{equation}
for any $t\geq0$, any $F$ with $F\Delta E_t^+\subset\subset M\backslash E_t^+$, proving that $E_t^+$ is outward optimising.\\
To prove strictly minimising, suppose $F$ contains $E_t^+$ and 
\begin{equation*}
|\partial E_t^+\cap A|-|\partial^*F\cap A|=\int_{F\backslash E_t^+}(g^{ij}-\nu_M^i\nu_M^j)K_{ij}.
\end{equation*}
Then by (\ref{1.17}), $\nabla u=0$ a.e. on $F\backslash E_t^+$. Since $F$ is also outward optimising, and the Lebesgue points of an outward optimising set form an open set in $M$, by a measure zero modification we may assume $F$ is open. Then $u$ is constant on each connected component of the open set $F\backslash\{u\leq t\}$. Since $M$ has no compact components, Lemma \ref{Conn} (i) means that no connected component of $F$ can have closure disjoint from $\bar{E}_t^+$, therefore $u=t$ on $F\backslash E_t$ and $F\subseteq E_t^+$. This proves that $E_t^+$ is strictly outward optimising.\\
\noindent(\textit{iii}) It is clear from part (\textit{ii}) and the definition of $E_t'$ that $E_t^{\prime}\subseteq E_t^+$. Assume $E_t^+$ is precompact. Then if 
\begin{equation*}
|E_t'\cap A|=|E_t^+\cap A|+\int_{E_t^+\backslash E_t'}(g^{ij}-\nu_M^i\nu_M^j)K_{ij},
\end{equation*} 
strict outward optimisation implies that $E_t^{\prime}=E_t^+$. Otherwise 
\begin{equation*}|\partial E_t^{\prime}\cap A|<|\partial E_t^+\cap A|+\int_{E_t^+\backslash E_t'}(g^{ij}-\nu_M^i\nu_M^j)K_{ij}, 
\end{equation*}
contradicting (\ref{1.17}).\\
\noindent(\textit{iv}) In view of (\textit{i}), we can use $E_t^+$ as a competitor to obtain
\begin{equation}\label{ppp}
|\partial E_t\cap A|\leq|\partial E_t^+\cap A|+\int\limits_{E_t^+\backslash E_t}(g^{ij}-\nu_M^i\nu_M^j)K_{ij}dx,
\end{equation}
for $t>0$, and for $t=0$ if $E_0$ happens to be outward optimising itself.
Since $E_t^+$ is precompact, strict inequality in (\ref{ppp}) would contradict $(\textit{iii})$, implying equality in (\ref{ppp}), which proves $(\textit{iv})$.\\
The proof for general $(U,\bar{\nu})$ satisfying (\ref{initial}) in $\tilde{\Omega}\subset M\times\mathbb{R}$ follows exactly as above. \end{outwardproof}
Outward Optimising Lemma \ref{outward1} implies that $\partial E_t$ satisfies the obstacle problem minimising ``area plus bulk energy $P$", with $E_t$ as the obstacle. This leads to a heuristic interpretation of the minimisation principle (\ref{Mvar}). Namely, as long as $E_t$ remains strictly outward optimising, it evolves by inverse null mean curvature, and when this condition is violated, $E_t$ jumps to $E_t'$ and continues. 
This implies that the null mean curvature is nonnegative on the weak solution after time zero. Furthermore, part $(\textit{iv})$ of Lemma \ref{outward1} implies that the monotonicity property
\begin{equation}
\dfrac{d}{dt}|\Sigma_t|+\int\limits_{E_t\backslash E_0}P=|\Sigma_t|
\end{equation}
derived in Lemma \ref{monoton}, also holds in the weak setting, as long as $\Sigma_t$ remains compact.\\

The outward optimising property also implies a stronger result for the surfaces foliating the jump region in Proposition \ref{jump}, namely we see that each $\tilde{\Sigma}_{X_0}$ is a smooth MOTS in $\tilde{\mathcal{K}}_{t_0}$.
\begin{Proposition}\label{MOTS1}
Each surface $\tilde{\Sigma}_{X_0}$ from Proposition \ref{jumpprop} in the foliation of the interior $\tilde{\mathcal{K}}_{t_0}$ of the jump region in $M\times\mathbb{R}$ is a smooth MOTS. \end{Proposition}
To prove Proposition \ref{MOTS1}, we require the following Lemma.
\begin{Lemma}\label{conv}\begin{equation}\label{gradU}
|\bar{\nabla}U_i|\to 0\quad\text{in  }L^1_{\text{loc}}(\tilde{\mathcal{K}}_{t_0}).
\end{equation}
\end{Lemma}
\begin{proof}
Recall $d$ defined by (\ref{ddd1}), consider a target point $X_0=(x_0,z_0)$ such that $z_0>2d+1$ and select a cutoff function $\phi\in C^2_c(\mathbb{R})$ such that $\phi\geq0$ and spt$\,\phi\subseteq[z_0-2d,z_0+2d]$. Then let $T_0=z_0-2d-1$, fix an arbitrary time $T>T_0$, and consider $T_0\leq t\leq T$ and $L\geq T+3+z_0+2d$. 

We wish to show that 
\begin{equation*}\liminf\limits_{i\to\infty}\int_{\tilde{\Sigma}^{i}_{t_{i}}\cap B^{M\times\mathbb{R}}_d(X_0)}|D(H+P)|^2<\infty.\end{equation*}
To this end, we calculate
\begin{align}\label{a}
 &\dfrac{d}{dt}\int_{\tilde{\Sigma}^{\varepsilon}_t}\phi(z) (H+P)^2\\
 \,\,=&\int_{\tilde{\Sigma}^{\varepsilon}_t}2\phi(H+P)\dfrac{\partial}{\partial t}(H+P)+(H+P)^2\dfrac{\partial\phi}{\partial z}\cdot\dfrac{\nu_{\varepsilon}}{H+P}+\phi H(H+P)\notag\\
=&-2\int_{\tilde{\Sigma}^{\varepsilon}_t}\phi\Biggl((H+P)\Delta\left(\frac{1}{H+P}\right)+|A|^2+\bar{Ric}(\nu_{\varepsilon},\nu_{\varepsilon})-\bar{\nabla}_{\nu_{\varepsilon}}P\notag\\
&+\dfrac{2D_i(H+P)}{H+P}K_{i\nu_{\varepsilon}}\Biggr)+(H+P)\dfrac{\partial\phi}{\partial z}\cdot\nu_{\varepsilon}+\phi H(H+P)\notag\\
=&\int_{\tilde{\Sigma}^{\varepsilon}_t}\phi\Biggl(-2\dfrac{|D(H+P)|^2}{(H+P)^2}-2|A|^2-2\bar{Ric}(\nu_{\varepsilon},\nu_{\varepsilon})+H(H+P)\notag\\
\notag&+2\bar{\nabla}_{\nu_{\varepsilon}}P-4\dfrac{D_i(H+P)}{H+P}K_{i\nu_{\varepsilon}}\Biggr)-2\dfrac{\phi}{\partial z}\cdot\dfrac{D(H+P)}{H+P}+(H+P)\dfrac{\partial\phi}{\partial z}\cdot\nu_{\varepsilon}
\end{align}
\vskip 0.1 true in
\noindent In view of the sup estimates  (\ref{supest1}) and (\ref{supestimate}) for $u_{\varepsilon}$, there is $R(T)>0$ depending only on the subsolution $v$ and $K_{ij}$ such that \\
\begin{equation*} 
 \tilde{\Sigma}^{\varepsilon}_t\cap(M\times\text{spt}\phi)\subseteq S(T):=(B_{R(T)}\backslash E_0)\times[z_0-2d,z_0+2d],\quad T_0\leq t\leq T.
\end{equation*}
The Outward Optimising Property \ref{outward1}, applied to $\tilde{E}^{\varepsilon}_t$ compared to the perturbation $\tilde{E}^{\varepsilon}_t\cup S(T)$, then provides the area estimate
\begin{equation}\label{areaest}
 |\tilde{\Sigma}^{\varepsilon}_t\cap(M\times\text{spt}\phi)|\leq C(T)+\int_{S(T)\backslash \tilde{E}^{\varepsilon}_t}P\leq C(T,\|K\|_{C^0}),\quad T_0\leq t\leq T.
\end{equation}
Together with the interior estimate (\ref{H+P estimate}), and the boundary gradient estimates for $u^{\varepsilon}$, this shows 
\begin{equation*}
|H+P|\leq C(T,\|K\|_{C^1})\quad\quad\text{ on }\tilde{\Sigma}^{\varepsilon}_t\cap(M\times\text{spt}\phi),\quad T_0\leq t\leq T.
\end{equation*}
It follows that 
\begin{equation*}
 \int_{\tilde{\Sigma}^{\varepsilon}_t}\phi|H|(H+P)+\phi(H+P)^2+|(H+P)\bar{\nabla}\phi\cdot\nu_{\varepsilon}|\leq C(T,\|K\|_{C^1}), \quad T_0\leq t\leq T.
\end{equation*}
We estimate the $D\phi$ and $K_{i\nu_{\varepsilon}}$ terms via
\begin{align*}
\left|2D\phi\cdot\dfrac{D(H+P)}{H+P}\right|&\leq2\dfrac{|D\phi|^2}{\phi}+\dfrac{\phi}{2}\dfrac{|D(H+P)|^2}{(H+P)^2}\leq C+\dfrac{\phi}{2}\dfrac{|D(H+P)|^2}{(H+P)^2},\\
\left|4\phi\dfrac{D_i(H+P)}{H+P}K_{i\nu_{\varepsilon}}\right|&\leq8\phi\|K\|_{C^0}^2+\dfrac{\phi}{2}\dfrac{|D(H+P)|^2}{(H+P)^2}.
\end{align*}
Thus (\ref{a}) becomes
\begin{equation}\label{monotone}
 \dfrac{d}{dt}\int_{\tilde{\Sigma}^{\varepsilon}_t}\phi( H+P)^2\leq\int_{\tilde{\Sigma}^{\varepsilon}_t}-\phi\dfrac{|D(H+P)|^2}{(H+P)^2}+C(T,\|K\|_{C^1}),
\end{equation}
and integrating gives
\begin{equation}
 \int_{T_0}^T\int_{\tilde{\Sigma}^{\varepsilon}_t\cap(M\times[z_0-2d,z_0+2d])}\dfrac{|D(H+P)|^2}{(H+P)^2}\leq C(T,\|K\|_{C^1}),
\end{equation}
using a $\phi$ such that $\phi=1$ on $[z_0-2d,z_0+2d]$.\\
Applying Fatou's Lemma, for any sequence $\varepsilon_i\rightarrow0$ we obtain
\begin{equation}\label{bb2}
 \liminf_{i\rightarrow\infty}\int_{\tilde{\Sigma}^{i}_t\cap(M\times[z_0-2d,z_0+2d])}\dfrac{|D(H+P)|^2}{(H+P)^2}<\infty, \quad\quad\text{a.e. }t\geq T_0.
\end{equation}
Now consider the subsequence $\varepsilon_{i_j}\to0$ from (\ref{subsequence1}) such that $\tilde{\Sigma}^{\,i_j}_ {t_{i_j}}\to \tilde{\Sigma}_0$ in $C^1(T\cap B^{n+1}_R(X_0))$, where $T=T_{X_0}\tilde{\Sigma}_{X_0}$. We write $i=i_j$ henceforth. Since (\ref{bb2}) only holds for a.e. $t\geq T_0$, it will take more work to argue that $\liminf\limits_{i\to\infty}\int_{\tilde{\Sigma}^{i}_{t_{i}}\cap B^{M\times\mathbb{R}}_d(X_0)}|D(H+P)|^2<\infty$. To this end, we pick a sequence $\hat{t}_i$ such that $\hat{t}_i\to t_0+\delta$ for some $|\delta|>0$, $|\hat{t}_i-t_i|\leq \varepsilon_i d$ and
\begin{equation}\label{l31}
\liminf_{i\to\infty}\int\limits_{\tilde{\Sigma}^i_{\hat{t}_i}\cap(M\times[z_0-2d,z_0+2d])}|D(H+P)|^2<\infty.
\end{equation}
Define $\hat{z}_i:=\dfrac{u_i(x_0)}{\varepsilon_i}-\dfrac{\hat{t}_i}{\varepsilon_i}$ and $\delta_i:=\hat{z}_i-z_0$. Then the fact that $\tilde{\Sigma}^i_{\hat{t}_i}$ is just a translation of $\tilde{\Sigma}^i_{t_i}$ by $\delta_i$ in the $z$-direction implies that 
\begin{equation*}
\int\limits_{\tilde{\Sigma}^i_{t_i}\cap B^{M\times\mathbb{R}}_d(X_0)}|D(H+P)|^2=\int\limits_{\tilde{\Sigma}^i_{\hat{t}_i}\cap B^{M\times\mathbb{R}}_d(x_0,z_0+\delta_i)}|D(H+P)|^2,
\end{equation*}
for each $i$. Furthermore, the condition $|\hat{t}_i-t_i|\leq \varepsilon_i d$  implies that $|\delta_i|=|\hat{z}_i-z_0|\leq d$, which ensures that $\tilde{\Sigma}^i_{\hat{t}_i}\cap B_d^{M\times\mathbb{R}}(x_0,z_0+\delta_i)\subset M\times[z_0-2d,z_0+2d]$, and thus from (\ref{l31}) we obtain that
\small\begin{equation*}
\liminf_{i\to\infty}\int_{\tilde{\Sigma}^i_{\hat{t}_i}\cap B^{M\times\mathbb{R}}_d(x_0,z_0+\delta_i)}|D(H+P)|^2\leq\int_{\tilde{\Sigma}^i_{\hat{t}_i}\cap (M\times[z_0-2d,z_0+2d])}|D(H+P)|^2<\infty,
\end{equation*}\normalsize
from which our desired estimate follows
\begin{equation}\label{l41}
\liminf_{i\to\infty}\int_{\tilde{\Sigma}^{i}_{t_{i}}\cap B^{M\times\mathbb{R}}_d(X_0)}|D(H+P)|^2<\infty.
\end{equation}
As in the proof of Lemma (\ref{foliation1}), the converging surfaces $\tilde{\Sigma}^i_{t_i}$ can be written locally, via the exponential map, as graphs of $C^{1,\alpha}_{\text{loc}}$ functions $w_i$ over the hyperplane $T$. This local $C^{1,\alpha}$ convergence of the hypersurfaces, together with the first variation of area formula and the Riesz Representation Theorem then implies that $H_{\tilde{\Sigma}_{_{X_0}}}$ exists weakly as a locally $L^1$ function, with the weak convergence
\begin{equation}\label{wkHconvergence}
\int_{\tilde{\Sigma}^i_{t_i}}H_{\tilde{\Sigma}^i_{t_i}}\nu_{\tilde{\Sigma}^i_{t_i}}\cdot X\to\int_{\tilde{\Sigma}_{_{X_0}}}H_{\tilde{\Sigma}_{_{X_0}}}\nu_{\tilde{\Sigma}_{_{X_0}}}\cdot X,\quad\quad X\in C^0_c(T(M\backslash E_0\times\mathbb{R})).
\end{equation}
Then by (\ref{l41}) and Rellich's theorem there exists a subsequence (again denoted by $i$) such that 
\begin{equation}
(H+P)_{\tilde{\Sigma}^{i}_{t_{i}}}\to (H+P)_{\tilde{\Sigma}_{_{X_0}}}\quad\quad\text{in }L^2(T\cap B^{n+1}_R(X_0))).
\end{equation}
Now the level-sets $\tilde{\Sigma}^i_{t_i}=\{U_i=t_i\}$ smoothly solve $(*)$ in $\Omega_i\times\mathbb{R}$, thus
\begin{equation*}
(H+P)_{\tilde{\Sigma}^i_{t_i}}=|\bar{\nabla} U_i|,
\end{equation*}
and
\begin{equation}\label{l51}
\int_{\tilde{\Sigma}^{i}_{t_{i}}}|\bar{\nabla} U_{i}|^2=\int_{\tilde{\Sigma}^{i}_{t_{i}}}(H+P)^2\to\int_{\tilde{\Sigma}_{_{X_0}}}(H+P)^2.
\end{equation}
To proceed, we consider the special behaviour of the solution in the jump region. Let us first consider the case where the limit surface $\tilde{\Sigma}_{X_0}$ given by Lemma \ref{foliation1} is not a vertical cylinder. Then it is a graph, which means that  $|\nabla\hat{u}_i|$ converges locally uniformly to something finite, and therefore that $|\nabla u_i|=\varepsilon_i|\nabla\hat{u}_i|$ converges locally uniformly to $0$.
In the other case the surface $\tilde{\Sigma}_{X_0}$ given by Lemma \ref{foliation1} is a vertical cylinder. We know from (\ref{l41}) and (\ref{l51}) that $|\bar{\nabla} U_{\bar{i}}|$ converges in $L^2$ to something finite. However this limit can only be zero since $U_i\to U$ locally uniformly, and $U$ is constant in the jump region (namely $U=t_0$ on $\tilde{\mathcal{K}}_{t_0}$ by hypothesis). Furthermore, since the local uniform convergence of $U_i\to t_0$ holds for the entire sequence $i$, we must have $L^2$ convergence of the entire sequence $|\bar{\nabla} U_i|$, which implies
\begin{equation*}
\int_{\tilde{\Sigma}^i_{t_i}}|\bar{\nabla}U_i|\to0.
\end{equation*}
\end{proof}
\begin{MOTS1Proof} Proposition \ref{jumpprop} and Lemma \ref{conv} imply that each surface $\tilde{\Sigma}_{X_0}$ in the jump region bounds a Caccioppoli set that minimises area plus bulk energy $\left(\bar{g}^{ij}-\tilde{\nu}^i\tilde{\nu}^j\right)K_{ij}$ in $\tilde{\mathcal{K}}_{t_0}$. To complete the proof of Proposition \ref{jumpprop}, it remains to show that each surface $\tilde{\Sigma}_{X_0}$ in $\tilde{\mathcal{K}}_{t_0}$ is in fact a smooth MOTS. To proceed, we use the connection between parametric and non-parametric variational problems, that follows from the relationship between a function $w\in BV_{\text{loc}}(\Omega)$ and its subgraph
\begin{equation}\label{subgraph}
W=\{(x,t)\in\Omega\times\mathbb{R}:t<w(x)\}.
\end{equation}
In particular, let $\varphi_W$ denote the characteristic function of the subgraph (\ref{subgraph}). Then Theorem 14.6 in [Gi] states
\begin{equation}\label{giusti}
\int_{\Omega}\sqrt{1+|\nabla w|^2}=\int_{\Omega\times\mathbb{R}}|\nabla\varphi_W|.
\end{equation}

In (\ref{subsequence1}) we established that at each point $X_0\in\tilde{\mathcal{K}}_{t_0}$, there exists a subsequence $i_j$ and a function $\hat{w}\in C^{1,\alpha}(\hat{B}_R(\hat{x}_0))$ such that 
\begin{equation*}
\hat{w}_{i_j}\to\hat{w}\quad\quad\text{in }C^1(\hat{B}_R(\hat{x}_0)),
\end{equation*}
where $\hat{B}_R(\hat{x}_0):=\hat{T}\cap B^{n+1}_R(\hat{x}_0)$, where
\begin{equation}\label{Graph}
\text{graph}(\hat{w})=\hat{\Sigma}_{X_0}=\exp_q^{-1}\left(\tilde{\Sigma}_{X_0}\cap B^{M\times\mathbb{R}}_{R}(X_0)\right)
\end{equation}
Then Lemma \ref{minimising} establishes that each surface $\tilde{\Sigma}_{X_0}$ bounds a Caccioppoli set $E$ in $\tilde{\mathcal{K}}_{t_0}$ that minimises area plus bulk energy $(\bar{g}^{ij}-\tilde{\nu}^i\tilde{\nu}^j)K_{ij}$ in $\tilde{\mathcal{K}}_{t_0}$, where by construction $\tilde{\nu}$ is the outward unit normal vector to the relative boundary $\partial \tilde{E}\cap \tilde{\mathcal{K}}_{t_0}$.\\

Therefore, writing the Caccioppoli set $E$ locally as the subgraph of $w:=\exp_q^{-1}(\hat{w})$, we find from (\ref{giusti}) that $w$ minimises the functional
\begin{align*}
\mathring{J}^{B_R(x_0)}_{\tilde{\nu}}(w):=&\int_{B_R(x_0)}\sqrt{1+|\bar{\nabla} w|^2}dx\\
\notag&+\int_{B_R(x_0)}\int_0^{w(x)}\text{tr}_MK_{ij}(x,\tau)-K_{ij}(x,\tau)\tilde{\nu}^i(x,\tau)\tilde{\nu}^j(x,\tau)d\tau dx
\end{align*}
in $B_R(x_0):=\exp_q(\hat{B}_R\hat{x}_0))$, whose Euler-Lagrange equation is the MOTS equation
\begin{equation}\label{minsurfeqn}
 \text{div}\left(\dfrac{\bar{\nabla} w}{\sqrt{1+|\bar{\nabla} w|^2}}\right)+\left(\bar{g}^{ij}-\tilde{\nu}^i\tilde{\nu}^j\right)K_{ij}=0,
\end{equation}
and by construction $\tilde{\nu}=\dfrac{(\bar{\nabla}w,-1)}{\sqrt{1+|\bar{\nabla}w|^2}}$. The left hand side of $(\ref{minsurfeqn})$ is an elliptic operator of the form
\begin{equation*}
 Aw=a^{ij}(\bar{\nabla} w)\left(\bar{\nabla}_i\bar{\nabla}_jw+\sqrt{1+|\bar{\nabla} w|^2}K_{ij}\right),
\end{equation*}
where 
\begin{equation*}
 a^{ij}(p):=\dfrac{1}{\sqrt{1+|p|^2}}\left(\bar{g}^{ij}-\dfrac{p^ip^j}{1+|p|^2}\right).
\end{equation*}
Since $w\in C^{1,\alpha} (B_R(x_0))$, $a^{ij}\in C^{0,\alpha}(B_R(x_0))$, $Aw$ is strictly elliptic on $B_R(x_0)$. Schauder theory then implies that $w\in C^{2,\alpha}(B_R(x_0))$, and by bootstrapping further we obtain $w\in C^{\infty}(B_R(x_0))$. Using a suitable partition of unity, we obtain that each surface $\tilde{\Sigma}_{X_0}$ is a smooth MOTS in $\tilde{\mathcal{K}}_{t_0}$. 
\end{MOTS1Proof}

\section{Existence of weak solutions}\label{2.5}
In this section we use the normal vector field $\tilde{\nu}$ of the hypersurface foliation of the jump region $\tilde{\mathcal{K}}_{t_0}$ from Proposition \ref{jumpprop} to extend $\bar{\nu}=\frac{\bar{\nabla}U}{|\bar{\nabla}U|}$ across the jump region, thereby constructing a globally defined normal vector field $\bar{\nu}$. Existence of weak solutions is then proven by taking the limit of the translating graphs $\tilde{\Sigma}^{\varepsilon}_t$, using Compactness Property \ref{compactness}. 

\begin{Theorem}[Existence of weak solutions]\label{altIMCF}
Let $(M^{n+1},g,K)$ be a complete, connected, asymptotically flat initial data set without boundary, that satisfies $\text{tr}_MK\geq0$. Then for any nonempty, precompact, open set $E_0\subset M$ with $C^2$ boundary, there exists a weak solution of $(**)$ with initial condition $E_0$. 
\end{Theorem}
\begin{proof}
Let $U$ be the limit of $U_{\varepsilon}$ as given by (\ref{Udefn1}). We construct the vertical cylinders $\tilde{\Sigma}_t:=\partial \{U<t\}$ and $\tilde{\Sigma}_t^+:=\partial \{U>t\}$ with local uniform $C^{1,\alpha}$ bounds and unit normal vector field $\nu$ with local  $C^{0,\alpha}$ bounds. Then using Theorem (\ref{compactness}), we show that $\{U<t\}$ minimises $\mathcal{J}_{U,\bar{\nu}}$ in $(M\backslash E_0)\times\mathbb{R}$ for each $t$, where $\bar{\nu}$ is extended in the jump regions $\tilde{\mathcal{K}}_{t_0}$ by the normal vector field $\tilde{\nu}$ to the family of smooth MOTS $\{\tilde{\Sigma}_{_{X_0}}\}_{X_0\in\tilde{\mathcal{K}}_{t_0}}$. \\ 

\noindent \textit{i)} In the case where $\tilde{\Sigma}_t=\tilde{\Sigma}_t^+$, the surface $\tilde{\Sigma}_t$ is constructed by fixing a point $X_0=(x_0,z_0)\in \tilde{\Sigma}_t$ and considering the sequence of times $t_i$ such that $X_0\in \tilde{\Sigma}^i_{t_i}$ for each $i$. It then follows exactly as in the proof of Lemma \ref{foliation1} that $\tilde{\Sigma}^i_{t_i}$ converges locally uniformly to $\tilde{\Sigma}_t$. Since $\tilde{\Sigma}_t=\tilde{\Sigma}_t^+$ is a vertical cylinder, convergence holds for the full sequence, and the unit normal vector field $\tilde{\nu}$ is equal to $\frac{\bar{\nabla}U}{|\bar{\nabla}U|}$. \\

\noindent\textit{ii)} We use a slightly different pointwise approach to construct $\tilde{\Sigma}_t$ and $\tilde{\Sigma}_t^+$ when $\tilde{\Sigma}_t\neq\tilde{\Sigma}_t^+$. To this end, let $X_0\in\tilde{\Sigma}_{t_0}^+$ at a jump time $t_0$. Since there are only countably many such $t_0$, there exists a sequence of points $X_i\in \tilde{\Sigma}_{t_i}$ with $t_i>t_0$, satisfying $\lim\limits_{i\to\infty}X_i=X_0$ and $\lim\limits_{i\to\infty}t_i=t_0$. For $i\gg1$ large enough, we can assume that $\tilde{\Sigma}_{t_i}=\tilde{\Sigma}_{t_i}^+$, and  as above each surface piece $\tilde{\Sigma}_{t_i}\cap B^{M\times\mathbb{R}}_R(X_i)$ can therefore be written via the exponential map (denoted by the hat superscript) as the graph of a $C^{1,\alpha}$ function $\hat{w}_i$ over $T_{\hat{X}_i}\hat{\Sigma}_{t_i}$, where
\begin{equation*}
 \hat{\Sigma}_{t_i}:=\exp_{X_i}^{-1}(\tilde{\Sigma}_{t_i}\cap B^{M\times\mathbb{R}}_d(X_i)).
 \end{equation*}
Now consider the sequence $\hat{\nu}_i$ of normal vectors to $\hat{\Sigma}_{t_i}$ at $\hat{X}_i$. Since the $\hat{\nu}_i(\hat{X}_i)$ are uniformly bounded in $C^{0,\alpha}$, there exists a subsequence $\hat{\nu}_{i_j}$ and a unit vector field $\hat{\nu}$ such that $\hat{\nu}_{i_j}\to\hat{\nu}$ uniformly on $B^{n+2}_{\hat{R}}(\hat{X}_i)$. Let $\hat{T}$ denote the hyperplane containing $\hat{X}_0$ and orthogonal to $\hat{\nu}(\hat{X}_0)$. For $i\gg1$ large enough, we can then write each surface $\hat{\Sigma}_{t_i}$  locally as the graph of a $C^{1,\alpha}$ function $\hat{w}_i$ over $\hat{T}\cap B^{n+2}_{\hat{R}}(\hat{X}_0)$. 
By Arzela-Ascoli, there exists a further subsequence $\hat{w}_{i_j}$ and a $C^{1,\alpha}$ function $\hat{w}:\hat{T}\cap B^{n+1}_{\hat{R}}(\hat{X}_i)$ such that 
\begin{equation*}
\hat{w}_i\to\hat{w}\quad\quad\text{in }C^1(\hat{T}\cap B^{n+1}_{\hat{R}}(\hat{X}_i)),
\end{equation*}
where $\hat{X}_0\in\text{graph}\,\hat{w}$ and $\hat{T}=T_{X}\text{graph}\,\hat{w}$. In order to recognise graph$\,\hat{w}$ as a piece of $\hat{\Sigma}_{t_0}^+$ and $\hat{T}$ as $T_{\hat{X}_0}\hat{\Sigma}_{t_0}^+$, we consider a point $Y\in$ graph$\,\hat{w}$. Then there exists a sequence $Y_i\in\text{graph}\,\hat{w}_i\subset \hat{\Sigma}_{t_i}$ such that $Y_i\to Y$, and thus $\hat{U}(Y_i)=t_i$ implies $\hat{U}(Y)=t_0$, where $\hat{U}:=U\circ \exp.$

In order to obtain a contradiction, assume that $Y\in \hat{E}_{t_0}^+$. Then there exists $\delta>0$ such that $B^{n+2}_{\delta}(Y)\in \hat{E}_{t_0}^+$, however this contradicts the fact that $Y_i\in \hat{\Sigma}_{t_i}$ for $t_i>t_0$. Thus we deduce that graph$\,\hat{w}\in \hat{\Sigma}_{t_0}^+$ as required. The case where $X_0\in \tilde{\Sigma}_{t_0}$ for $\tilde{\Sigma}_{t_0}\neq \Sigma_{t_0}^+$ follows analogously.\\

In Proposition \ref{jumpprop} we constructed a family of surfaces $\{\tilde{\Sigma}_{_{X_0}}\}_{X_0\in\tilde{\mathcal{K}}_{t_0}}$ foliating the jump region $\tilde{\mathcal{K}}_{t_0}$ of $U$. This foliation has a $C^{0,\alpha}_{\text{loc}}$ normal vector field $\tilde{\nu}$, which extends the vector field of the surfaces $\tilde{\Sigma}_{t_0}$ and $\tilde{\Sigma}_{t_0}^+$ as a calibration across the jump region at jump times  $t_0$ via the definition
\begin{equation*}\label{nudefn}\quad
\bar{\nu}(x):=\begin{cases}\begin{aligned}
&\frac{\bar{\nabla}U}{|\bar{\nabla}U|}(x)\quad\quad\quad\,\,\,  \text{if }x\in\Sigma_t\text{ at regular times }t,\\
&\tilde{\nu}\quad\quad\quad\quad\quad\quad\,\,\,\, \text{if }x\in\tilde{\mathcal{K}}_{t_0} \text{ at a jump time }t_0,\\
&\lim_{i\to\infty}\frac{\bar{\nabla}U}{|\bar{\nabla}U|}(x_i)\,\,\, \text{ if }x\in\tilde{\Sigma}_{t_0}, \text{ where } x_i\in \tilde{\Sigma}_{t_i}\text{ for }x_i\to x, t_i\nearrow t_0,\\
&\lim_{i\to\infty}\frac{\bar{\nabla}U}{|\bar{\nabla}U|}(x_i)\,\,\, \text{ if }x\in\tilde{\Sigma}_{t_0}^+, \text{ where }x_i\in \tilde{\Sigma}_{t_i} \text{ for }x_i\to x, t_i\searrow t_0.
\end{aligned}
\end{cases}
\end{equation*}
This global interpretation of the normal vector field $\bar{\nu}$ in $M\backslash\bar{E_0}\times\mathbb{R}$ means the functional $\mathcal{J}_{U,\bar{\nu}}$ is well defined on $M\backslash\bar{E_0}\times\mathbb{R}$, and it follows from Compactness Property \ref{compactness} that the sets $\{U<t\}$ minimises $\mathcal{J}_{U,\bar{\nu}}$ in $M\backslash\bar{E}_0\times\mathbb{R}$ for each $t$. The result then follows from Lemma \ref{minimising}. \end{proof}

\section{Applications of weak solutions}\label{2.6}
In this section we highlight the natural applications of weak solutions of $(**)$ to the existence theory for MOTS and to the theory of weak solutions of IMCF.\\

\textbf{MOTS.} The one-sided variational principal associated with outward optimisation implies that the solution must jump at $t=0$ wherever the null mean curvature of $\Sigma_0=\partial E_0$ is strictly negative. Together with Proposition \ref{MOTS1}, this implies the following existence theorem for MOTS in initial data sets $(M,g,K)$ containing an outer trapped surface $\Sigma_0$ such that $\theta^+_{\Sigma_0}<0$.
\begin{Proposition}[Existence of smooth MOTS]\label{MOTS}
Let $(M^{n+1},g,K)$ be an asymptotically flat initial data set of dimension $n+1\leq 7$ satisfying $\text{tr}_MK\geq0$, and let $E_0$ be any nonempty, precompact, smooth open set in $M$ satisfying $\theta^+_{\partial E_0}<0$ with respect to the unit normal pointing out of $E_0$. 
Then the level set $\partial \{u>0\}$ of the weak solution $(U(x,z)=u(x),\bar{\nu})$ of $(**)$ is a smooth MOTS in $M\backslash E_0$. \end{Proposition}
Proposition \ref{MOTS} highlights the natural application of the theory of weak solutions to $(**)$ as a variational-type approach to constructing marginally outer trapped surfaces. Let $\kappa$ denote the largest eigenvalue of $K$ with respect to $g$ across $M\backslash E_0$. Then in the special case where the outermost MOTS in $M\backslash E_0$ satisfies
\begin{equation}
|\partial E|\leq|\partial F|-n\kappa\mathcal{L}^{n+1}(E\backslash F),
\end{equation}
for every closed MOTS $\partial F$ in $M\backslash E_0$, then the weak solution to $(**)$ in Proposition \ref{MOTS} will jump to the outermost MOTS $\Sigma=\partial\{u>0\}$ at $t=0$. In general, given any initial data, and any initial condition $E_0$ satisfying $\theta^+_{\partial E_0}<0$, the surface can't jump beyond the outmost MOTS at $t=0$, and the MOTS $\partial\{u>0\}$ is an inner barrier for the outermost MOTS $\Sigma$ in $M\backslash E_0$.

We compare Proposition \ref{MOTS}  to the following existence theorem combining \cite{AM09} and \cite{E09}, as stated in \cite{AEM09}.
\begin{Theorem}[\cite{AM09,E09}]\label{AEM}
Let $(M,g,K)$ be an initial data set of dimension $n+1\leq 7$ and let $\Omega\subset M$ be a connected bounded open subset with smooth embedded boundary $\partial\Omega$. Assume this boundary consists of two non-empty closed hypersurfaces $\partial_+\Omega$ and $\partial_-\Omega$, possibly consisting of several components, such that 
\begin{equation}
H_{\partial_+\Omega}-\text{tr}_{\partial_+\Omega}K>0\quad\text{and }H_{\partial_+\Omega}+\text{tr}_{\partial_+\Omega}K>0,
\end{equation}
where the mean curvature scalar is computed as the tangential divergence of the unit normal vector field that is pointing out of $\Omega$. Then there exists a smooth closed embedded hypersurface $\Sigma\subset\Omega$ homologous to $\partial_-\Omega$ such that $H_{\Sigma}+\text{tr}_{\Sigma}K=0$ (where $H_{\Sigma}$ is computed with respect to the unit normal pointing towards $\partial_-\Omega$).
$\Sigma$ is $\lambda$-almost minimising in $\Omega$ for $\lambda=2(n+1)\kappa$, where $\kappa$ denotes the largest eigenvalue of $K$ with respect to $g$ across $\Omega$.
\end{Theorem}

Here \textit{$\lambda$-almost minimising in $\Omega$} means that the surface $\Sigma$ arises as (a relative boundary of) a subset $E\subset\Omega$ with perimeter $\Sigma$ in $\Omega$ such that 
\begin{equation}\label{C-almost}
|\partial E\cap W|\leq |\partial F\cap W|+\lambda\mathcal{L}^{n+1}(E\Delta F),
\end{equation}
for every $F\subset\Omega$ such that $E\Delta F\subset\subset W\subset\subset\Omega$. A detailed analysis of such almost-minimising boundaries is carried out in \cite{DS93}. We say that the set $E$ is $\lambda$-almost minimising \textit{on the outside/inside} in $\Omega$ if $E$ satisfies (\ref{C-almost}) for every $F$ such that $E\Delta F\subset\subset W$, where $F\subseteq E$, $F\supseteq E$ respectively.\\

\textbf{Remark.} Since minimising area plus bulk energy $P$ is a stronger condition than $\lambda$-almost minimising, weak solutions $(U(x,z)=u(x),\bar{\nu})$ of $(**)$ satisfy the following $\lambda$-almost minimising properties, for $\lambda:=n\kappa$, where $\kappa$ denotes the size of the largest eigenvalue of $K$ on $M\backslash E_0$:
\begin{itemize}
\item[$(i)$] The smooth MOTS $\partial\{u>0\}$ of Proposition \ref{MOTS} is $\lambda$-almost minimising.
\item[$(ii)$] The sets $\tilde{E}_t=\{U<t\}$ and $\{U\leq t\}$ are $\lambda$-almost minimising on the outside for each $t>0$, $t\geq0$ respectively. 
\item[$(iii)$] The surfaces $\tilde{\Sigma}_{X_0}$ foliating the interior $\tilde{\mathcal{K}}_{t_0}$ of the jump region  are $\lambda$-almost minimising in $\tilde{\mathcal{K}}_{t_0}$.
\end{itemize}
\vskip 0.2 true in
\textbf{IMCF.} In this section we discuss the above results in the context of the work of Huisken and Ilmanen \cite{HI01} on inverse mean curvature flow. In particular, when applied to the special case $K\equiv0$, Definition \ref{alternative} provides a new perspective on weak solutions to inverse mean curvature flow, and the work of sections \ref{sec:jump} and \ref{jjump} carries over to prove the analogous results for the jump region.

The degenerate elliptic equation
\begin{equation}\tag{$\star$}
 \mbox{div}_M\left( \frac{\nabla u}{|\nabla u|}\right) = |\nabla u|
\end{equation}
describes inverse mean curvature flow of the level-sets of the scalar function $u:M\to\mathbb{R}$ wherever $|\nabla u|\neq0$. In \cite{HI01}, Huisken and Ilmanen define a locally Lipschitz function $ u\in C^{0,1}_{loc}(M)$ to be a weak solution of ($\star$) with initial condition $E_0$ if $E_0=\{u<0\}$ and
\begin{equation}\label{J}
 J_u(u)\leq J_u(v),\quad\quad \text{where }\,\,\,J_u(v)=J_u^A(v):=\int_A|\nabla v|+v|\nabla u|,
\end{equation}
for every locally Lipschitz function $v$ with $\{v\neq u\}\subset\subset M\backslash E_0$, where the integral is performed over any compact set $A\supseteq\{u\neq v\}$. 
They then showed that it follows that $u$ is a weak solution if and only if the open set $E_t:=\{u<t\}$ minimizes the parametric energy functional
\begin{equation}\label{JJ}
J^A_u(F)=|\partial^*F\cap A|-\int_{F\cap A}|\nabla u|,
\end{equation}
in $M\backslash E_0$ for each $t>0$. \begin{Theorem}[Existence of weak solutions, \cite{HI01}]\label{HI existence}
Let $M$ be a complete, connected Riemannian manifold without boundary. Suppose there exists a proper, locally Lipschitz, weak subsolution $v$ of (\ref{J}) with a precompact initial condition.
Then for any nonempty, precompact, smooth open set $E_0$ in $M$, there exists a proper, locally Lipschitz weak solution $u$ of $(\star)$ with initial condition $E_0$, which is unique on $M\backslash E_0$. \end{Theorem}

In \cite{He01}, Heidusch proved optimal $C_{\text{loc}}^{1,1}$ regularity for the level sets $N_t=\partial\{u<t\}$ and $N_t^+=\partial\{u>t\}$ of the weak solution. The theory of weak solutions to inverse mean curvature flow as laid out in \cite{HI01}, however, does not include an analysis of the interior of jump regions. Applying Proposition \ref{MOTS1} in this special case where $K\equiv0$, we obtain a foliation of the interior of the jump region $\{u=t_0\}\times\mathbb{R}$ by area minimising hypersurfaces, a result which was left open in \cite{HI01}.

\begin{Corollary}\label{jump}
 Let $u$ be the weak solution of inverse mean curvature flow given by Theorem \ref{HI existence}. At a jump time $t_0$, the interior $\tilde{\mathcal{K}}_{t_0}$ of the region $\{u=t_0\}\times\mathbb{R}$ is foliated by smooth area minimising surfaces, each of which is either a vertical cylinder or a smooth graph over an open subset of $\tilde{\mathcal{K}}_{t_0}$.
\end{Corollary}

We can then utilise the jump region hypersurfaces of Corollary \ref{jump} to present a new perspective on weak solutions of inverse mean curvature flow. In particular, by instead considering the weak solution to be a family of hypersurfaces one dimension higher in $M\times\mathbb{R}$, we can ask for the functional $J_U$, defined by (\ref{JJ}), to be minimised \textit{everywhere} in $(M\backslash \bar{E}_0)\times\mathbb{R}$, and obtain the following richer notion of weak solution.

\begin{Definition}[Alternative weak formulation]\label{alternative}
Let $u$ be the unique, locally Lipschitz weak solution to $(\star)$ on $M\backslash E_0$ given by Theorem \ref{HI existence}, and define the locally Lipschitz function $U(x,z):=u(x)$ on $(M\backslash E_0)\times\mathbb{R}$. The weak solution to $(\star)$ is defined to be the pair $(U,\nu)$, where $\nu$ is a unit length, translation invariant extension of $\dfrac{\bar{\nabla}U}{|\bar{\nabla}U|}$ in the jump regions such that at each point $x\in\tilde{\mathcal{K}}_{t_0}$, $\nu(x)$ is the normal vector to a $C^{1,\alpha}$ hypersurface passing through $x$, which bounds a Caccioppoli set that minimises $J_{U,\nu}$ in $\tilde{\mathcal{K}}_{t_0}$. \end{Definition}
 
We obtain the following weak existence result as a corollary of Theorem \ref{altIMCF}.
\begin{Corollary}[Existence of weak solutions]\label{altIMCF2}
Let $M$ be a complete, connected Riemannian $n$-manifold without boundary. Suppose there exists a proper, locally Lipschitz, weak subsolution of (\ref{J}) with a precompact initial condition. 
Then for any nonempty, precompact, smooth open set $E_0$ in $M$, there exists a weak solution satisfying Definition \ref{alternative} in $M\backslash E_0\times\mathbb{R}$ with initial condition $E_0$. 
\end{Corollary}

\end{document}